\documentclass[12pt]{article}

\usepackage{graphicx}
\usepackage{enumerate}
\usepackage{amssymb}
\usepackage{amsmath}
\usepackage{mathrsfs}
\usepackage{latexsym}
\usepackage{psfrag}
\usepackage{mathrsfs}
\usepackage{verbatim}
\usepackage{xcolor}
\usepackage{comment}

\newcommand{\qed}{\hfill $\square$\\
\vspace{0.1cm}}

\newtheorem{theorem}{Theorem}[section]
\newtheorem{definition}[theorem]{Definition}
\newtheorem{lemma}[theorem]{Lemma}

\newtheorem{proposition}[theorem]{Proposition}

\newtheorem{conjecture}{Conjecture}[section]

\newenvironment{proof}{\noindent{\em Proof.}}{\qed}

\begin{document}
%\baselineskip=1.25\baselineskip

\title{Ramsey Properties for $V$-shaped Posets in the Boolean Lattices}

\author{
Hong-Bin Chen
\thanks{Department of Applied Mathematics, National Chung Hsing University, Taichung 40227, Taiwan
{\tt Email:andanchen@gmail.com} supported by MOST-110-2115-M-005-006 - }
\and
Wei-Han Chen
\thanks{Department of Applied Mathematics, National Chung Hsing University, Taichung 40227, Taiwan
{\tt Email:sevenak47@gmail.com}}
\and
Yen-Jen Cheng
\thanks{Department of Mathematics, National Taiwan Normal University, Taipei 116, Taiwan
{\tt Email:yjc7755@gmail.com} supportes by MOST- 110-2811-M-A49-505}
\and
Wei-Tian Li
\thanks{Department of Applied Mathematics, National Chung Hsing University, Taichung 40227, Taiwan
{\tt Email:weitianli@nchu.edu.tw} supported by MOST-109-2115-M-005 -002 -}
\and
Chia-An Liu
\thanks{Department of Mathematics, Soochow University, Taipei 11102, Taiwan
{\tt Email:liuchiaan8@gmail.com}
supported by MOST-109-2115-M-031-006-MY2
}
}

\date{\small \today}

\maketitle

\begin{abstract}
Given posets $\mathbf{P}_1,\mathbf{P}_2,\ldots,\mathbf{P}_k$, 
let the {\em Boolean Ramsey number} $R(\mathbf{P}_1,\mathbf{P}_2,\ldots,\mathbf{P}_k)$ be the minimum number $n$ such that no matter how we color the elements in the Boolean lattice $\mathbf{B}_n$ with $k$ colors, 
there always exists a poset $\mathbf{P}_i$ contained in $\mathbf{B}_n$ whose elements are all colored with $i$. 
This function was first introduced by Axenovich and Walzer~\cite{AW}. Recently, many results on determining $R(\mathbf{B}_m,\mathbf{B}_n)$ have been published. In this paper, we will study the function  $R(\mathbf{P}_1,\mathbf{P}_2,\ldots,\mathbf{P}_k)$ for each $\mathbf{P}_i$'s being the $V$-shaped poset. That is, a poset obtained by identifying the minimal elements of two chains. 

Another major result presented in the paper is to determine the minimal posets $\mathbf{Q}$ contained in $\mathbf{B}_n$, when $R(\mathbf{P}_1,\mathbf{P}_2,\ldots,\mathbf{P}_k)=n$ is determined, having the Ramsey property described in the previous paragraph.

In addition, we define the {\em Boolean rainbow Ramsey number} $RR(\mathbf{P},\mathbf{Q})$ the minimum number $n$ such that when arbitrarily  coloring the elements in $\mathbf{B}_n$, there always exists either a monochromatic $\mathbf{P}$ or a rainbow $\mathbf{Q}$ contained in $\mathbf{B}_n$. 
The upper bound for $RR(\mathbf{P},\mathbf{A}_k)$ was given by Chang, Li, Gerbner, Methuku, Nagy, Patkos, and Vizer for general poset $\mathbf{P}$ and $k$-element antichain $\mathbf{A}_k$. We study the function for $\mathbf{P}$ being the $V$-shaped posets in this paper as well.
\end{abstract}

{\bf keywords:} posets, Boolean lattices, Ramsey number, rainbow coloring.

\section{Introduction}

Ramsey theory is a research topic on finding the substructure of a large mathematics structure when partitioning the set of the objects in the large structure into subsets.  
The structures that have been studied including the set systems (hypergraphs), the set of integers, the Euclidean spaces, and the general posets~\cite{AGLM,CFS1,CFS2,GRS,T99}. 
The investigation of the Ramsey property on posets was initialled by N\u{e}set\u{r}il and R\"{o}dl~\cite{NR} and Kierstead and Trotter~\cite{KT}. 
As a convention, we use the capital letters $A$, $B,\ldots$, the boldface letters $\mathbf{A}$, $\mathbf{B},\ldots$, and the scribble letters $\mathcal{A}$, $\mathcal{B},\ldots$ to denote sets, posets, and families of sets in this paper, respectively.
A {\em poset} $\mathbf{P}=(S,\le)$ is a set $S$ associated with a partial ordering on the pairs of elements in $S$. Two elements $x$ and $y$ in $S$ are said to be {\em comparable} if either $x\le_\mathbf{P}y$ or $y\le_\mathbf{P}x$ holds; otherwise, they are {\em incomparable}.  We say a poset $\mathbf{Q} $ contains a poset $\mathbf{P}$ as a {\em subposet} (or $\mathbf{Q} $ contains a copy of $\mathbf{P}$) if there exists an injection $\phi:S_\mathbf{P}\rightarrow S_\mathbf{Q}$ such that $x\le_\mathbf{P} y$ in $\mathbf{P}$ if and only if  $\phi(x)\le_\mathbf{Q} \phi(y) $ in $\mathbf{Q}$.
While if the above injection only preserves the partial ordering  
($\phi:S_\mathbf{P}\rightarrow S_\mathbf{Q}$ and $x\le_\mathbf{P} y$ in $\mathbf{P}$ implies  $\phi(x)\le_\mathbf{Q} \phi(y) $ in $\mathbf{Q}$), 
then we say $\mathbf{Q}$ contains $\mathbf{P}$ as a {\em weak subposet}.
Without specific mentioning, the subposets in the paper will be the former version.
Two posets are {\em isomorphic} if each of them contains the other one as a subposet. 

Let $p$ be a poset parameter such as cardinality, height, width, etc. 
Kierstead and Trotter~\cite{KT} studied the function $f_p(n)$ that is the least integer so that 
for every poset $\mathbf{P}$ with $p(\mathbf{P})=n$ there exists a poset $\mathbf{Q}$ with $p(\mathbf{Q})=f_p(n)$ such that whenever the elements of $\mathbf{Q}$ are colored with red or blue, it always contains a copy of $\mathbf{P}$ whose elements are all of the same color.
Along this line, Axenovich and Walzer~\cite{AW} studied the following poset Ramsey numbers for the Boolean lattices. A {\em Boolean lattice} is a poset whose underlying set consists of all subsets of $[n]:=\{1,2,\ldots,n\}$ as the elements and the inclusion relation among sets as the partial order relation, denoted as $\mathbf{B}_n=(2^{[n]},\subseteq)$.

\begin{definition}
Given posets $\mathbf{P}_1,\mathbf{P}_2,\ldots,\mathbf{P}_k$, let $R(\mathbf{P}_1,\mathbf{P}_2,\ldots,\mathbf{P}_k)$ be the minimum number $n$ such that 
for every $k$-coloring $c:2^{[n]}\rightarrow [k]$, $\mathbf{B}_n$ contains a monochromatic $\mathbf{P}_i$ of color $i$ as a subposet. If $\mathbf{P}_1=\cdots=\mathbf{P}_k=\mathbf{P}$, we may denote $R(\mathbf{P}_1,\mathbf{P}_2,\ldots,\mathbf{P}_k)$ with $R_k(\mathbf{P})$ for short.
\end{definition}
The {\em 2-dimension} of a poset $\mathbf{P}$, defined by Trotter~\cite{T75} and denoted by $\dim_2(\mathbf{P})$, is the least integer $n$ for which $\mathbf{B}_n$ contains $\mathbf{P}$ as a subposet. Note that the value of $R(\mathbf{P}_1,\mathbf{P}_2)$ for $\mathbf{P}_1=\mathbf{P}_2=\mathbf{B}_n$ is indeed the same as $f_{\dim_2}(n)$.
Thus, it provides the motivation to study $R(\mathbf{B}_n,\mathbf{B}_n)$. 
Besides, the problem of determining the value of $R(\mathbf{B}_m,\mathbf{B}_n)$ for general $m$ and $n$ seems to be an interesting problem on its own right, and it has drawn a lot of attention. There are several papers~\cite{AW,BP,GMT,LT} containing the results on this aspect.
Meanwhile, the weak subposet version of poset Ramsey number was proposed and investigated by Cox and Stolee~\cite{CS}. 
Moreover, Falgas-Ravry, Markstr\"{o}m, Treglown, and Zhao used the probability method to study the Ramsey problems in the Boolean lattices in a recent paper~\cite{FMTZ}.

In the following we present our results in this paper. 
Since one can study the function $R(\mathbf{P}_1,\mathbf{P}_2,\ldots,\mathbf{P}_k)$ for other poset parameters, we call the one in the paper the {\em Boolean Ramsey number} (This term is also used in the article of Cox and Stolee~\cite{CS}). 
Let $m$ and $n$ be two integers with $n\ge m\ge 1$, and let $\mathbf{V}_{m,n}$ denote the poset on $m+n+1$ elements $x$, $y_1\ldots, y_m$, and $z_1\ldots, z_n$ with the partial order relations $x<y_1<y_2<\cdots <y_m$ and $x<z_1<z_2<\cdots <z_n$. The paper determines the exact values of the Boolean Ramsey number for $\mathbf{V}_{m,n}$ with some given conditions on $m$ and $n$.

\begin{theorem}[identical type]\label{IT}
Let $k\ge 1$ and $n\ge m\ge 1$. If we have 
(1) $m=n=1$, or
(2) $m<n$,
then \[R_k(\mathbf{V}_{m,n})=nk+1.\]
\end{theorem}

\begin{theorem}[mixed type]\label{MT}
Let $m,n\ge 1$. We have 
\[R(\mathbf{V}_{m,m},\mathbf{V}_{n,n})=m+n+1.\]
\end{theorem}

From the definition of $f_{\dim_2}(n)$, 
it is nature to ask the question: Given a poset $\mathbf{P}$ with  
$\dim_2(\mathbf{P})=n$, what is the poset $\mathbf{Q}$ with $\dim_2(\mathbf{Q})=f(n)$
that contains a monochromatic $\mathbf{P}$ as a subposet for every 2-coloring of $\mathbf{Q}$?
Analogously, despite that $R(\mathbf{P}_1,\mathbf{P}_2,\ldots,\mathbf{P}_k)=n$ implies that $\mathbf{B}_n$ is a poset with the smallest 2-dimension  containing a $\mathbf{P}_i$ of color $i$ for every $k$-coloring of $\mathbf{B}_n$, we are interested in finding other posets $\mathbf{Q}$ with $\dim_2(\mathbf{Q})=n$ having the same property. 
Let $n=R(\mathbf{P}_1,\mathbf{P}_2,\ldots,\mathbf{P}_k)$.
A subposet $\mathbf{Q}\subset\mathbf{B}_n$ is called {\em minimal $(\mathbf{P}_1,\mathbf{P}_2,\ldots,\mathbf{P}_k)$-Ramsey} (or {\em minimal $(\mathbf{P};k)$-Ramsey} when $\mathbf{P}_i=\mathbf{P}$ for all $i$) if every $k$-coloring of $\mathbf{Q}$ contains a monochromatic $\mathbf{P}_i$ of color $i$ for some $i$, but for any subposet $\mathbf{Q}'\subset \mathbf{Q}$ with $|\mathbf{Q}'|=|\mathbf{Q}|-1$,  there exists a $k$-coloring that avoids any
monochromatic $P_i$ of color $i$. 
We determine the minimal posets for the following two cases.

\begin{theorem}\label{V1k}
The poset 
\[\mathbf{B}_{k+1}-\{[k+1]\}\] is the unique minimal $(\mathbf{V}_{1,1};k)$-Ramsey poset.
\end{theorem}

If the underlying Boolean lattice $\mathbf{B}_n$ is given, we denote $X^c=[n]-X$ the complementary set of $X\in \mathbf{B}_n$.

\begin{theorem}\label{V2V1}
A poset $\mathbf{P}$ is minimal $(\mathbf{V}_{1,1},\mathbf{V}_{2,2})$-Ramsey if and only if it is isomorphic to 
\[\mathbf{B}_4-\{S_1,S_2,[4]\}\]
for some pairs of nonempty $S_1$ and $S_2$ in $\mathbf{B}_4-\{[4]\}$ with $S_1=S_2^c$, or $|S_1|=|S_2|=2$.
\end{theorem}

By dropping the restriction of the number of colors in the colorings, we consider the following function of two posets. Here we say a poset $\mathbf{P}$ is {\em rainbow} if any two distinct  elements of $\mathbf{P}$ have different colors. 

\begin{definition}
Given posets $\mathbf{P}$ and $\mathbf{Q}$, we define the Boolean rainbow Ramsey number $RR(\mathbf{P},\mathbf{Q})$ to be the smallest $n$ such that for every coloring of $\mathbf{B}_n$, there exists either a monochromatic $\mathbf{P}$ or a rainbow $\mathbf{Q}$.
\end{definition}

Rainbow Ramsey numbers for graphs have been intensively studied. For a recent survey, see~\cite{FMO}. A result of Johnston, Lu, and Milans~\cite{JLM} shows that  $RR(\mathbf{B}_m,\mathbf{B}_n)$ is finite for all $m,n\in\mathbb{N}$, hence $RR(\mathbf{P},\mathbf{Q})$ exists for all $\mathbf{P}$ and $\mathbf{Q}$.
In~\cite{CCLL}, the authors studied $RR(\mathbf{P},\mathbf{Q})$ for $\mathbf{P}$ and $\mathbf{Q}$ being the combinations of the antichains, the chains, and the Boolean lattices. The relations of the Boolean Ramsey numbers and the Boolean rainbow Ramsey numbers were studied by Chang, Li, Gerbner, Methuku, Nagy, Patkos, and Vizer in~\cite{C7}. Also, they gave an upper bound for $RR(\mathbf{P},\mathbf{A}_k)$ for general poset $\mathbf{P}$ and the antichain $\mathbf{A}_k$ on $k$ elements. 

\begin{theorem}\label{RRPA}{\rm\cite{C7}}
Given an integer $k\ge 2$, let $m_k=\min\{m: \binom{m}{\lfloor m/2\rfloor }\ge k\}$. 
For any poset $\mathbf{P}$ we have 
\[RR(\mathbf{P},\mathbf{A}_k)\le \lfloor(k-1)\lambda^*_{max}(\mathbf{P})\rfloor+m_k.\]
Moreover, if $\mathbf{P}$ is not a chain on one or two elements, then  
we have 
\[
RR(\mathbf{P},\mathbf{A}_3)\le \lfloor2\lambda^*_{max}(\mathbf{P})\rfloor+2.
\] 
\end{theorem}

The value 
\[
\lambda^*_{max}(\mathbf{P})=\sup\left\{
\sum_{F\in\mathcal{F}}\binom{n}{|F|}^{-1}\mid\mathcal{F}\mbox{ is a }\mathbf{P}\mbox{-free family of subsets of }[n].
\right\}
\]
is an important parameter of a poset $\mathbf{P}$ in forbidden subposet theory. Meroueh~\cite{M} proved that it is finite for every poset $\mathbf{P}$ but the evaluation of $\lambda^*_{max}(\mathbf{P})$ is difficult in practical. For more background on forbidden subposet theory, we refer the readers to~\cite{GP,GL}.

We have the result on the Boolean rainbow Ramsey number below.

\begin{theorem}\label{RRJA}
For $k\ge 2$ and $1 \le m < n$, we have
\[
RR(\mathbf{V}_{m,n},\mathbf{A}_k) = n(k-1)+2. 
\]
\end{theorem}

The following are the organizations of the remaining sections of this paper. We first  determine the Boolean Ramsey numbers for $V$-shaped posets as stated in Theorem~\ref{IT} and Theorem~\ref{MT} in next section. 
In Section 3, we show the minimal poset properties as stated in Theorem~\ref{V1k} and Theorem~\ref{V2V1}. 
The proof of Theorem~\ref{RRJA} will be presented in Section 4 together with an improvement of Theorem~\ref{RRPA} (for $k=2$). We put some comments and open problems in the last section. 

\section{The Boolean Ramsey Numbers}

To prove Theorem~\ref{IT}, we introduce the following two lemmas that will be used in the proof. In fact, they will play important roles in many proofs.

\begin{lemma}\label{chain}
Let $\mathbf{C}$ be a chain in the Boolean lattice $\mathbf{B}_n$ such that at least one of  $\varnothing$ and $[n]$ is not an element in $\mathbf{C}$.
Then $\mathbf{B}_n - \mathbf{C}$ contains  $\mathbf{B}_{n-1}$ as a subposet.
\end{lemma}
	
\begin{proof}
Suppose that $X$ is the smallest element in $\mathbf{C}$ and $X \neq \varnothing$.
Without loss of generality, we may assume $n\in X$. Thus $n \in Y$ for all $Y \in \mathbf{C}$.
So every subset of $[n-1]$ is an element in  $\mathbf{B}_n-\mathbf{C}$.
This implies $\mathbf{B}_n - \mathbf{C}$ contains $\mathbf{B}_{n-1}$ as a subposet.

If $\varnothing\in\mathbf{C}$ then $[n]\not\in\mathbf{C}$,
and we may assume $n\not\in Y$ for all $Y\in \mathbf{C}$.
Define an order-preserving injective mapping $\phi$ from $\mathbf{B}_{n-1}$ to $\mathbf{B}_n-\mathbf{C}$ by $\phi(X)=X\cup \{n\}$ for every $X\in\mathbf{B}_{n-1}$. 
This shows that $\mathbf{B}_n - \mathbf{C}$ contains a copy of $\mathbf{B}_{n-1}$.
\end{proof}

One can remove an antichain instead of a chain in the Boolean lattice and still obtain the same result. This property was pointed out and named {\em the antichain lemma} in~\cite{AW}. 
To make the paper self-contained, we give the proof of the lemma. 

\begin{lemma}{\rm \cite{AW}}\label{antichain}
Let $\mathbf{A}$ be an antichain in the Boolean lattice $\mathbf{B}_n$.
Then $\mathbf{B}_n- \mathbf{A}$ contains  $\mathbf{B}_{n-1}$ as a subposet.
\end{lemma}
	
\begin{proof}
Let $\mathbf{A}$ be an antichain in $\mathbf{B}_n$.
Define $\phi$ from $\mathbf{B}_{n-1}$ to $\mathbf{B}_n - \mathbf{A}$ by
\[
\phi(X) = \left \{
\begin{array}{ll}
X, & \mbox{ if }X\mbox{ does not contain any }Y\in\mathbf{A}\mbox{ as a subset,}\\
X \cup \{n\}, & \mbox{ if }X\mbox{ contains some }Y\in\mathbf{A}\mbox{ as a subset.}
\end{array}
\right.	
\]
Then we show that $\phi$ is injective and order-preserving.

For two distinct elements $X_1, X_2\in \mathbf{B}_{n-1}$,
if both of them contain some elements in $\mathbf{A}$ as the subsets or
neither of them contain an element in $\mathbf{A}$, 
then it is clear that $\phi(X_1)\neq \phi(X_2)$.
Suppose one of them, say $X_1$, contains an element in $\mathbf{A}$, and the other does not.
Then $\phi(X_1)=X_1\cup\{n\}\neq X_2=\phi(X_2)$.
So $\phi$ is injective.

Next we show that $\phi$ is order-preserving.
Suppose $X_1\subseteq X_2$ for some $X_1,X_2\in\mathbf{B}_{n-1}$.
If $Y\subseteq X_1$ for some $Y\in\mathbf{A}$, then $Y\subseteq X_2$.
Hence $\phi(X_1)=(X_1\cup\{n\})\subseteq (X_2\cup\{n\})=\phi(X_2)$.
If $Y\not\subseteq X_1$ for any $Y\in\mathbf{A}$,
then $\phi(X_1)=X_1\subseteq X_2\subseteq \phi(X_2)$.
Assume both $ X_1\not\subseteq X_2 $ and $ X_2\not\subseteq X_1 $.
Then clearly, $X_i\not \subseteq (X_j\cup\{n\} )$ and $(X_i\cup\{n\})\not \subseteq (X_j\cup\{n\} )$ for $\{i,j\}=\{1,2\}$. Therefore,
$ \phi(X_1)\not\subseteq \phi(X_2) $ and $\phi(X_2)\not\subseteq \phi(X_1) $ hold.
As a conclusion, $\phi$ is an order-preserving mapping and $\mathbf{B}_n-\mathbf{A}$ contains $\mathbf{B}_{n-1}$ as a subposet.
\end{proof}

Let us introduce a notation that will appear in many proofs. Given $\mathbf{B}_n$, we define 
\[\mathbf{B}_{n}^{i,j} : = \{ X \subseteq [n]\mid  i \in X,\, j \not\in X \}.\]
Note that $\mathbf{B}_{n}^{i,j}$ is isomorphic to $\mathbf{B}_{n-2}$.
\bigskip

\noindent{\em Proof of Theorem~\ref{IT}.}
The lower bound can be obtained by using the same coloring method for both conditions (1) and (2).
For $X\in \mathbf{B}_{nk}$, 
color $X$ with $\lfloor\frac{|X|}{n} \rfloor+1$ for $X\neq [nk]$ and color $[nk]$ with any $i\in [k]$. 
Note that the poset $\mathbf{V}_{m,n}$ contains a chain on $n+1$ elements.    
In this coloring, a color class either consists of subsets of $n$ different sizes, or $[nk]$ belongs to the color class and it consists of subsets of $n+1$ different sizes.  
However, $[nk]$ cannot be any element in $\mathbf{V}_{m,n}$ since it is greater than any other elements in $\mathbf{B}_{nk}$. Thus the lower bound holds.

For the upper bounds, we first show  $R_k(\mathbf{V}_{1,1})\le k+1$ by induction on $k$.
For $k=1$, it is obvious that a single color $\mathbf{B}_2$ contains a monochromatic $\mathbf{V}_{1,1}$ whose elements are  $\varnothing$, $\{1\}$ and $\{2\}$.
Suppose the upper bound holds for some integer $k-1\ge 0$.
Consider any $k$-coloring $c$ of $\mathbf{B}_{k+1}$.
Without loss of generality, we assume $c(\varnothing)=1$.
If there are two incomparable subsets $X$ and $Y$ colored with $1$, then we have a  monochromatic $\mathbf{V}_{1,1}$ that consists of $\varnothing$, $X$ and $Y$.
Thus, we can assume all subsets of color 1 form a chain.
Let $\mathbf{C}$ be the chain formed by all subsets of color 1 except $[k+1]$ if $c([k+1])=1$.
By Lemma~\ref{chain},  $\mathbf{B}_{k+1}-\mathbf{C}$ contains a copy of $\mathbf{B}_{k}$.
Pick such a copy of $\mathbf{B}_{k}$
from $\mathbf{B}_{k+1}-\mathbf{C}$.
Either there are only $k-1$ colors on the elements in it, or we have $c([k+1])=1$ and there are $k$ colors on the elements in it.
The latter case also implies that $[k+1]$ is the maximal element in the copy of $\mathbf{B}_{k}$ we picked.
If $c([k+1])=1$, we replace it with any other color in $\{2,\ldots, k\}$ to obtain a coloring of $\mathbf{B}_k$ with only $k-1$ colors.
Now by the inductive hypothesis, there exists a monochromatic $\mathbf{V}_{1,1}$ in $\mathbf{B}_k$.
Since $[k+1]$ cannot be any element in the monochromatic $\mathbf{V}_{1,1}$ we just got, the $\mathbf{V}_{1,1}$ actually has appeared in $\mathbf{B}_{k+1}$ under the original coloring $c$.
Therefore the upper bound holds for $m=n=1$ and  all $k\ge 1$.

Next we show $R_k (\mathbf{V}_{m,n}) \le nk+1$. We use induction on $k$ again.
For $k = 1$, there is only one color on $\mathbf{B}_{n+1}$.
Note that the subsets $[1],[2],\ldots,[n]$ form a chain on $n$ elements, and their complementary sets $[n]^c,[n-1]^c,\ldots,[1]^c$ also form a chain on $n$ elements. 
Since the two chains are incomparable, 
all the $2n$ sets with $\varnothing$ form the poset $\mathbf{V}_{n,n}$ which contains $\mathbf{V}_{m,n}$ as a subposet.
Suppose the upper bound holds for some integer $k-1 \ge 1$, that is $R_{k-1} (\mathbf{V}_{m,n}) \le n(k-1)+1$.
Let $c'$ be any $k$-coloring of $\mathbf{B}_{nk+1}$, and let $\mathcal{F}_\varnothing$ be the family of subsets colored with $c'(\varnothing)$.
Now we consider two cases:
$\mathcal{F}_\varnothing - \{\varnothing,[nk+1]\}$ contains a chain on $n$ elements, or
any chain in $\mathcal{F}_\varnothing- \{\varnothing,[nk+1]\}$ contains at most $n-1$ elements.
Recall that $\mathbf{B}_{nk+1}^{i,j}$ is isomorphic to $\mathbf{B}_{nk-1}$, 
and observe that every $X\in \mathbf{B}_{nk+1}^{i,j}$
and every $Y\in \mathbf{B}_{nk+1}^{j,i}$ are incomparable.

Suppose that $\mathcal{F}_\varnothing-\{\varnothing,[nk+1]\}$ contains a chain $\mathbf{C}$ on $n$ elements. Note that the monochromatic $\mathbf{C}$ in $\mathcal{F}_\varnothing- \{\varnothing,[nk+1]\}$ is contained in some $\mathbf{B}_{nk+1}^{i,j}$ where $i$ is an element in the smallest set in $\mathbf{C}$ and $j$ is an element not in the largest set in $\mathbf{C}$.
If $\mathbf{B}_{nk+1}^{j,i}\cap \mathcal{F}_\varnothing$ contains a chain $\mathbf{C}'$ on $m$ elements,
then $\varnothing$ together with elements in  $\mathbf{C}$ and $\mathbf{C}'$ form a monochromatic $\mathbf{V}_{m,n}$ of color 1, and we are done.
Otherwise, any chain in $\mathbf{B}_{nk+1}^{j,i}\cap \mathcal{F}_\varnothing$ contains at most $m-1$ elements.
With a result of Mirsty~\cite{MIR} on the poset decomposition,  we can partition the family $\mathbf{B}_{nk+1}^{j,i}\cap \mathcal{F}_\varnothing$ into at most $m-1$ antichains.
By removing these antichains from $\mathbf{B}_{nk+1}^{j,i}$ and applying Lemma \ref{antichain} at most $m-1$ times,
we see that $\mathbf{B}_{nk+1}^{j,i}-\mathcal{F}_\varnothing$ contains $\mathbf{B}_{nk-1-(m-1)}$ as a subposet. 
Since there are only $k-1$ colors on the elements in $\mathbf{B}_{nk+1}^{j,i}- \mathcal{F}_\varnothing$ and $nk-1-(m-1) \ge n(k-1)+1$, there exists a monochromatic $\mathbf{V}_{m,n}$ in $\mathbf{B}_{nk+1}^{j,i}$ by inductive hypothesis.

Now if any chain in $\mathcal{F}_\varnothing- \{\varnothing,[nk+1]\}$ contains at most $n-1$ subsets, then we can partition $\mathcal{F}_\varnothing- \{[nk+1]\}$ into at most $n$ antichains.
By removing these antichains and applying Lemma~\ref{antichain} at most $n$ times, we conclude that $\mathbf{B}_{nk+1} - (\mathcal{F}_\varnothing -\{[nk+1]\})$ contains a copy of $\mathbf{B}_{n(k-1)+1}$.
Once $c'([nk+1])=c'(\varnothing)$ and the previous $\mathbf{B}_{n(k-1)+1}$ has $[nk+1]$ as its element, we replace the color of $[nk+1]$ with any color other than $c'(\varnothing)$.
Thus, there are only $k-1$ colors on the elements in $\mathbf{B}_{nk+1} - (\mathcal{F}_\varnothing- \{[nk+1]\})$, 
and at most $k-1$ colors on the elements in the copy of $\mathbf{B}_{n(k-1)+1}$. By inductive hypothesis, there exists a monochromatic $\mathbf{V}_{m,n}$ in the copy of $\mathbf{B}_{n(k-1)+1}$ contained in $\mathbf{B}_{nk+1} - (\mathcal{F}_\varnothing- \{[nk+1]\})$.
As before, this monochromatic $\mathbf{V}_{m,n}$ does not contain $[nk+1]$ as its element. So it already exists in the original coloring $c'$ of $\mathbf{B}_{nk+1}$.
\qed

The content of next lemma is the core of the proof of Theorem~\ref{MT}. Nevertheless, we think it can be applied to the investigation of other Boolean Ramsey numbers, so we write it as an independent lemma. 

\begin{lemma}\label{VMN} 
Let $c$ be a 2-coloring of $\mathbf{B}_{m+n+1}$. If there exists $c(\{i\}^c)=c(\varnothing)$ for some $i\in[m+n+1]$, 
then there exists either a monochromatic $\mathbf{V}_{m,m}$ of color 1 or a monochromatic $\mathbf{V}_{n,n}$ of color 2.
\end{lemma}

\begin{proof}
The statement holds for $m=n=1$ since for any 2-coloring of $\mathbf{B}_3$, there exists a monochromatic $\mathbf{V}_{1,1}$ by Theorem~\ref{IT}. 
Suppose $m+n\ge 3$. 
Given a 2-coloring $c$ of $\mathbf{B}_{m+n+1}$, 
let $\mathcal{F}_\varnothing$ be the family of  subsets of color $c(\varnothing)$.
 
\noindent{\em Case 1.} There exist $\{i\}^c,\{j\}^c\in \mathcal{F}_\varnothing$ with $i\neq j$. 

First consider $m=1$ ($n=1$ is similar). 
If $c(\{i\}^c)=c(\{j\}^c)=c(\varnothing)=1$, then there is nothing to prove. So we consider  $c(\{i\}^c)=c(\{j\}^c)=c(\varnothing)=2$. 
For $\mathbf{B}_{n+2}^{i,j}\cap\mathcal{F}_\varnothing$ and  $\mathbf{B}_{n+2}^{j,i}\cap\mathcal{F}_\varnothing$, if both of them 
contain a chain on $n$ elements, then we pick such a chain from each of them. 
The two chains together with $\varnothing$ form a $\mathbf{V}_{n,n}$ of color 2.
So, we may assume that every chain in  $\mathbf{B}_{n+2}^{i,j}\cap \mathcal{F}_\varnothing$  contains at most $n-1$ elements. 
Moreover, since $c(\{j\}^c)=2$, every chain in 
$\mathbf{B}_{n+2}^{i,j}\cap(\mathcal{F}_\varnothing-\{\{j\}^c\})$ contains at most $n-2$ elements.
By Lemma~\ref{antichain}, 
$\mathbf{B}_{n+2}^{i,j}-(\mathcal{F}_\varnothing- \{\{j\}^c\})$ contains a copy of  $\mathbf{B}_{2}$ and all the elements but not necessary the maximal element (it could be $\{j\}^c$) in this $\mathbf{B}_{2}$ are of color 1. Thus, we have a $\mathbf{V}_{1,1}$ of color 1.

The proof of the condition both $m\ge 2$ and $n\ge 2$ is similar. Suppose that $c(\{i\}^c)=c(\{j\}^c)=c(\varnothing)=1$. 
As before, we may assume that $\mathbf{B}_{m+n+1}^{i,j}\cap \mathcal{F}_\varnothing$ does not contain a chain on $m$ elements, and every chain in 
$\mathbf{B}_{m+n+1}^{i,j}\cap(\mathcal{F}_\varnothing-\{\{j\}^c\})$ contains at most $m-2$ elements.
By Lemma~\ref{antichain}, 
$\mathbf{B}_{m+n+1}^{i,j}-(\mathcal{F}_\varnothing- \{\{j\}^c\})$ contains a copy of  $\mathbf{B}_{n+1}$ and all elements but not necessary the maximal element in this $\mathbf{B}_{n+1}$ are of color 2. Thus, we have a $\mathbf{V}_{n,n}$ of color 2.
If $c(\varnothing)=2$, then we can exchange $m$ and $n$ in the above argument. 

\noindent{\em Case 2.} There exists exactly one $\{i\}^c\in\mathcal{F}_\varnothing$.

Suppose $c(\{i\}^c)=c(\varnothing)=1$.
Again, for $\mathbf{B}_{m+n+1}^{i,j}\cap\mathcal{F}_\varnothing$ and  $\mathbf{B}_{m+n+1}^{j,i}\cap\mathcal{F}_\varnothing$, if there is no $\mathbf{V}_{m,m}$ of color 1, then 
at least one of them cannot contain a chain on $m$ elements. If it is $\mathbf{B}_{m+n+1}^{j,i}\cap\mathcal{F}_\varnothing$, then the rest of the proof is the same as the previous case. So we assume that every chain in  
$\mathbf{B}_{m+n+1}^{i,j}\cap\mathcal{F}_\varnothing$ contains at most $m-1$ elements. Note that $c(\{j\}^c)=2$. 
So we can only conclude that $\mathbf{B}_{m+n+1}^{i,j}- (\mathcal{F}_\varnothing-\{\{j\}^c\})$ contains a $\mathbf{B}_{n}$ of color 2. 
If $n=1$, then we can pick the minimal element of this $\mathbf{B}_1$ as $Y$, and pick $y\in [m+n+1]-(Y\cup\{j\})$.
Then $Y$, $\{j\}^c$, and $\{y\}^c$ form a  $\mathbf{V}_{1,1}$ of color 2. 
For $n\ge2$, since $\mathbf{B}_{n}$ contains $\mathbf{V}_{n-1,n-1}$, let $Y$, $W_1,W_2,\ldots,W_{n-1}$ and $Z_1,Z_2\ldots,Z_{n-1}$ be the sets forming $\mathbf{V}_{n-1,n-1}$ in $\mathbf{B}_{m+n+1}^{i,j}- (\mathcal{F}_\varnothing-\{\{j\}^c\})$. That is, $Y\subset W_1\subset W_2\subset\cdots\subset W_{n-1}$, $Y\subset Z_1\subset Z_2\subset\cdots\subset Z_{n-1}$, 
$W_p\not\subseteq Z_q$, and $Z_q\not\subseteq W_p$ for all $p,q\in[1,n-1]$.
Pick $w\in W_1-Z_{n-1}$ and $z\in Z_1-W_{n-1}$. 
Then we have $W_{n-1}\subset \{z\}^c$, 
$Z_{n-1}\subset \{w\}^c$, and   $W_p\not\subset\{z\}^c$ and $Z_q\not\subset\{w\}^c$ for all $p,q\in[1,n-1]$. 
Since $c(\{w\}^c)=(\{z\}^c)=2$, the sets $Y$, $W_p$'s,  $Z_q$'s, $\{w\}^c$ $\{z\}^c$ form a $\mathbf{V}_{n,n}$ of color 2.
The proof of $c(\{i\}^c)=c(\varnothing)=2$ can be obtained by exchanging $m$ and $n$ in the argument as well. So the proof is completed.
\end{proof}

\noindent{\em Proof of Theorem~\ref{MT}.} 
The lower bound is again obtained by simply coloring each set $X\in\mathbf{B}_{m+n}$ with 1 for $0\le |X|<m$ and with $2$ for $m\le |X|\le m+n$.

The upper bound will be proved by induction on $m+n$.
The theorem is true when $m+n=2$.
Let $m+n\ge 3$. Suppose that $c$ is a $2$-coloring of $\mathbf{B}_{m+n+1}$, and let  $\mathcal{F}_\varnothing$ be the family of subsets colored with $c(\varnothing)$. 
By Lemma~\ref{VMN}, if there exists some $\{i\}^c\in\mathcal{F}_\varnothing$, then we are done. 
Otherwise, say $c(\varnothing)=1$ and $c(\{i\}^c)=2$ for all $i\in[m+n+1]$. Then $\mathbf{B}_{m+n+1}-\{\{i\}^c\mid 1\le i\le m+n+1\}$ contains a $\mathbf{B}_{m+n}$. 
By inductive hypothesis, either there is a $\mathbf{V}_{m,m}$ of color 1 or a  $\mathbf{V}_{n-1,n-1}$ of color 2 in this $\mathbf{B}_{m+n}$. For the latter case, as the last part of Lemma~\ref{VMN}, we can concatenate this $\mathbf{V}_{n-1,n-1}$ with two sets $\{i\}^c$ and $\{j\}^c$ to obtain a  $\mathbf{V}_{n,n}$ of color 2.\qed

\section{Minimal Posets}

In Theorem~\ref{V1k} and Theorem~\ref{V2V1}, the colorings we gave to establish the lower bounds for the Boolean Ramsey numbers are all obtained by trivially coloring sets of the same sizes with one color.
In other words, the elements in the Boolean lattices are colored layer by layer.
To show Theorem~\ref{V1k} and Theorem~\ref{V2V1}, we will provide the colorings for the subposets of the minimal $(\mathbf{P}_1,\mathbf{P}_2,\ldots,\mathbf{P}_k)$-Ramsey poset avoiding the monochromatic $P_i$'s.
These coloring methods are somewhat more interesting than the previous layer colorings.   
\bigskip

\noindent{\em Proof of Theorem~\ref{V1k}.} We have proved that  $R_k(\mathbf{V}_{1,1})=k+1$. To show that 
$\mathbf{B}_{k+1}-\{[k+1]\}$ is minimal $(\mathbf{V}_{1,1};k)$-Ramsey, we construct a $k$-coloring of $\mathbf{B}_{k+1}-\{S\}$ not containing monochromatic $\mathbf{V}_{1,1}$ for every proper $S\subset [k+1]$.

If $S=\varnothing$, then there are only $k$ different sizes of sets in $\mathbf{B}_{k+1}-\{S\}$. Thus, we can color each set of size $i$ with $i$ for $1\le i\le k$ to avoid a monochromatic $\mathbf{V}_{1,1}$ in $\mathbf{B}_{k+1}$. 
For $S\neq\varnothing$, without loss of generality, let $S=[s]$.  We color the sets in the following way:
\[
c(X)=\left\{
\begin{array}{ll}
|X|,     &S\subset X,\\
|X|+1,     &S\not\subset X\mbox{ and }|X|\le k-1, \\
i, &X=[k+1]-\{i\}\mbox{ for }1\le i\le s.
\end{array}
\right.
\]
For this coloring, consider any three sets of color $i$. If $i\le s$, then the sizes of the three sets must be either $i-1$ or $k$. Note that there is only one set of size $k$ colored with $i$. So any three sets of color $i$ cannot form a $\mathbf{V}_{1,1}$ for $i\le s$. Else if $i\ge s+1$, then the size of the three sets must be $i-1$ or $i$. In this case, any two sets of size $i$ both contain $S$ as a subset, so does their intersection. On the other hand, a set of size $i-1$ colored with $i$ does not contain $S$. 
Thus, it is impossible to find three sets of color $i$ forming the poset $\mathbf{V}_{1,1}$ for $i\ge s+1$. So the proof is completed.
\qed

\noindent{\em Proof of Theorem~\ref{V2V1}.}
We first prove that $\mathbf{B}_{4}-\{S_1,S_2,[4]\}$ contains  either a $\mathbf{V}_{1,1}$ of color 1 or $\mathbf{V}_{2,2}$ of color 2 for every 2-coloring if 
$S_1=S_2^c$, or $|S_1|=|S_2|=2$.
Without loss of generality,  assume that $\{S_1,S_2\}$ is one of $\{[1],[1]^c\}$,  $\{[2],[2]^c\}$, and \{[2],\{1,3\}\}.
Consider any 2-coloring $c$ of $\mathbf{B}_{4}-\{S_1,S_2,[4]\}$. 
Let $\mathcal{F}_\varnothing$ be the collection of subsets of color $c(\varnothing)$. In our proof, we always assume that the coloring $c$ gives no $\mathbf{V}_{1,1}$ of color 1, but then there is a $\mathbf{V}_{2,2}$ of color 2.

\noindent{\em Case 1.} $\{S_1,S_2\}=\{[1],[1]^c\}$.

First suppose $c(\varnothing)=1$. Then $\mathcal{F}_\varnothing$ is a chain. 
There are two types of maximal chains in $\mathbf{B}_4-\{[1],[1]^c,[4]\}$. One of them intersects with $\{\{2,3\},\{2,4\},\{3,4\}\}$ 
and the other intersects with  
$\{\{1,2\},\{1,3\},\{1,4\}\}$. 
See the Hasse diagram shown in Figure~\ref{B4SS} to have a better understanding. 
By symmetry, we may assume that 
$\mathcal{F}_\varnothing$ is either contained in the chain $\{
\varnothing,\{4\},\{3,4\},\{1,3,4\}\}$ (first type) or in the chain $\{
\varnothing,\{4\},\{1,4\},\{1,3,4\}\}$ (second type).
In either case, the sets $\{2\}$, $\{2,3\}$, $\{2,4\}$, $\{1,2,3\}$, and $\{1,2,4\}$ form a $\mathbf{V}_{2,2}$ of color 2.

Next suppose $c(\varnothing)=2$.
By the pigeonhole principle, there exist two sets of size 3 which have the same color, so we may assume $c(\{1,2,3\})=c(\{1,3,4\})$. \begin{description}
\item{{\em Subcase 1.1.}} $c(\{1,2,3\})=c(\{1,3,4\})=1$.

Since there is no $\mathbf{V}_{1,1}$ of color 1,  $c(\{3\})=c(\{1,3\})=2$.
If $c(\{1,2,4\})=1$, then $c(\{4\})=c(\{1,4\})=2$, otherwise we have a $\mathbf{V}_{1,1}$ of color 1. Thus, $\{4\}$, $\{1,4\}$ $\{3\},\{1,3\}$ and $\varnothing$ form a $\mathbf{V}_{2,2}$ of color 2. Now consider the case $c(\{1,2,4\})=2$. If any of the sets $\{2\}$, $\{4\}$, and $\{2,4\}$ is colored with 2, then it with $\{1,2,4\}$, $\{3\}$, $\{1,3\}$ and $\varnothing$ form a $\mathbf{V}_{2,2}$ of color 2. Otherwise, all the three sets $\{2\}$, $\{4\}$, and $\{2,4\}$ are colored with 1. Then we can find a $\mathbf{V}_{1,1}$ of color 1 consisting of $\{4\}$, $\{2,4\}$, and $\{1,3,4\}$. This contradicts our assumption of $c$.

\item{{\em Subcase 1.2.}} 
$c(\{1,2,3\})=c(\{1,3,4\})=2$.

For the three sets $\{2\}$, $\{1,2\}$ and $\{2,3\}$, since there is no $\mathbf{V}_{1,1}$ of color 1, at least one of them is colored with 2. Similarly, at least one of $\{4\}$, $\{1,4\}$ and $\{3,4\}$ is colored with 2. 
Then pick one set of color 2 from each of the above two triples. The two sets together with $\{1,2,3\}$, $\{1,3,4\}$ and $\varnothing$ form a $\mathbf{V}_{2,2}$ of color 2. 

\end{description}

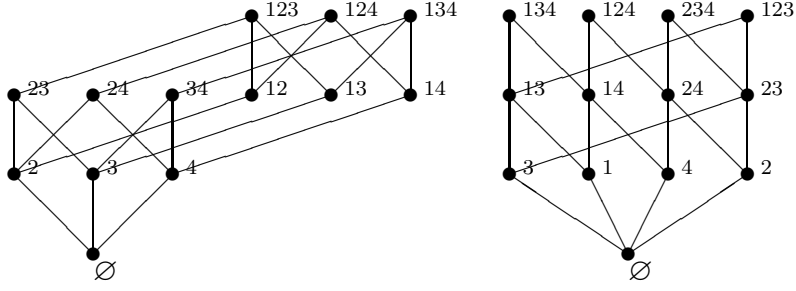
\begin{figure}[ht]
    \centering
\begin{picture}(160,100)
\put(30,0){$\varnothing$}
\put(30,10){\circle*{5}}
\put(0,40){\circle*{5}{\scriptsize 2}}
\put(30,40){\circle*{5}{\scriptsize 3}}
\put(60,40){\circle*{5}{\scriptsize 4}}
\put(0,70){\circle*{5}{\scriptsize 23}}
\put(30,70){\circle*{5}{\scriptsize 24}}
\put(60,70){\circle*{5}{\scriptsize 34}}
\put(90,70){\circle*{5}{\scriptsize 12}}
\put(120,70){\circle*{5}{\scriptsize 13}}
\put(150,70){\circle*{5}{\scriptsize 14}}
\put(90,100){\circle*{5}{\scriptsize 123}}
\put(120,100){\circle*{5}{\scriptsize 124}}
\put(150,100){\circle*{5}{\scriptsize 134}}
\put(0,40){\line(0,1){30}}
\put(0,40){\line(1,1){30}}
\put(0,40){\line(3,1){90}}
\put(30,40){\line(-1,1){30}}
\put(30,40){\line(1,1){30}}
\put(30,40){\line(3,1){90}}
\put(60,40){\line(0,1){30}}
\put(60,40){\line(-1,1){30}}
\put(60,40){\line(3,1){90}}

\put(90,70){\line(0,1){30}}
\put(90,70){\line(1,1){30}}
\put(0,70){\line(3,1){90}}
\put(120,70){\line(-1,1){30}}
\put(120,70){\line(1,1){30}}
\put(30,70){\line(3,1){90}}
\put(150,70){\line(0,1){30}}
\put(150,70){\line(-1,1){30}}
\put(60,70){\line(3,1){90}}

\put(30,10){\line(-1,1){30}}
\put(30,10){\line(0,1){30}}
\put(30,10){\line(1,1){30}}
\end{picture}
\qquad
\begin{picture}(100,100)
\put(0,40){\circle*{5}{\scriptsize 3}}
\put(30,40){\circle*{5}{\scriptsize 1}}
\put(60,40){\circle*{5}{\scriptsize 4}}
\put(90,40){\circle*{5}{\scriptsize 2}}
\put(0,70){\circle*{5}{\scriptsize 13}}
\put(30,70){\circle*{5}{\scriptsize 14}}
\put(60,70){\circle*{5}{\scriptsize 24}}
\put(90,70){\circle*{5}{\scriptsize 23}}
\put(0,100){\circle*{5}{\scriptsize 134}}
\put(30,100){\circle*{5}{\scriptsize 124}}
\put(60,100){\circle*{5}{\scriptsize 234}}
\put(90,100){\circle*{5}{\scriptsize 123}}
\put(45,10){\circle*{5}}
\put(45,0){$\varnothing$}
\put(45,10){\line(1,2){15}}
\put(45,10){\line(-1,2){15}}
\put(45,10){\line(3,2){45}}
\put(45,10){\line(-3,2){45}}

\put(0,40){\line(0,1){60}}
\put(30,40){\line(0,1){60}}
\put(60,40){\line(0,1){60}}
\put(90,40){\line(0,1){60}}

\put(0,40){\line(3,1){90}}
\put(30,40){\line(-1,1){30}}
\put(60,40){\line(-1,1){30}}
\put(90,40){\line(-1,1){30}}

\put(0,70){\line(3,1){90}}
\put(30,70){\line(-1,1){30}}
\put(60,70){\line(-1,1){30}}
\put(90,70){\line(-1,1){30}}

\end{picture}
    \caption{The Hasse diagrams of $\mathbf{B}_{4}-\{S,S^c,[4]\}$ for $S=[1]$ (left) and $S=[2]$ (right).}
    \label{B4SS}
\end{figure}

\noindent{\em Case 2. } $\{S_1,S_2\}=\{[2],[2]^c\}$.

Clearly, if $c(\varnothing)=1$ and there is no $\mathbf{V}_{1,1}$ of color 1, then $\mathcal{F}_\varnothing$ is a chain. 
Without loss of generality, we may assume  $\mathcal{F}_\varnothing\subset\{\varnothing,\{3\},\{1,3\},\{1,3,4\}\}$ or 
$\{\varnothing,\{3\},\{1,3\},\{1,2,3\}\}$.
In both cases, it is easy to find a $\mathbf{V}_{2,2}$ of color 2: e.g. $\{\{2\},\{2,4\},\{1,2,4\},\{2,3\},\{1,2,3\}\}$ for the former case, and
$\{\{4\},\{1,4\},\{1,3,4\},\{2,4\},\{2,3,4\}\}$ for the latter case.

Else, $c(\varnothing)=2$.  
If there is at most one set of size 3 colored with 2, then we may assume $c(\{1,2,3\})=c(\{1,2,4\})=c(\{1,3,4\})=1$. For the sets $\{3\}$, $\{1,3\}$, $\{4\}$, and $\{1,4\}$, either one of them has color 1 and we have a $\mathbf{V}_{1,1}$ of color 1, or all of them has color 2 and the four sets together with $\varnothing$ form a $\mathbf{V}_{2,2}$ of color 2.
Consider the case that at least two sets of size 3 colored with 2. 
Suppose we have two sets $X$ and $Y$ of size 3 colored with $2$ and $X\cap Y\in\{[2],[2]^c\}$, say $X=\{1,3,4\}$ and $Y=\{2,3,4\}$. Then either we have one of the two triples   $\{1\},\{1,3\},\{1,4\}$ and $\{2\},\{2,3\},\{2,4\}$ consisting of all sets of color 1, which form a $\mathbf{V}_{1,1}$ of color 1, or there is at least a set of color 2 in each triple and we have a $\mathbf{V}_{2,2}$ of color 2 formed by them with $X$, $Y$, and $\varnothing$. 
Another possibility is we have  $X\cap Y\not\in\{[2],[2]^c\}$, say $X=\{1,2,4\}$ and $Y=\{1,3,4\}$, 
and the other two sets of size 3, $\{1,2,3\}$ and $\{2,3,4\}$, are colored with 1. Now if $\{2\}$ or $\{3\}$ is colored with 1, then we have a $\mathbf{V}_{1,1}$ of color 1, otherwise we have a  $\mathbf{V}_{2,2}$ of color 2 consisting of $\varnothing$, $\{2\}$, $\{3\}$, $\{1,2,4\}$, and $\{1,3,4\}$.

\noindent{\em Case 3.}
$\{S_1,S_2\}=\{[2],\{1,3\}\}$.

As before, we first assume $c(\varnothing)=1$. Let $\mathcal{F}$ be the family consisting of the sets $\{2\}$, $\{3\}$, $\{4\}$, $\{2,3\}$, $\{2,4\}$, $\{3,4\}$, $\{1,2,3\}$, $\{1,2,4\}$, and $\{1,3,4\}$ in $\mathbf{B}_4-\{[1],\{1,3\},[4]\}$ and observe the Hasse diagrams formed by these sets in Figure~\ref{case3}.
Note that $\mathcal{F}_\varnothing$ is a chain. It is straightforward to see that $\mathcal{F}-\mathcal{F}_\varnothing$ contains a $\mathbf{V}_{2,2}$ of color 2.

\begin{figure}[ht]
    \centering
\begin{picture}(100,70)
\put(0,30){\circle*{5}{\scriptsize 23}}
\put(0,60){\circle*{5}{\scriptsize 123}}
\put(0,0){\circle*{5}{\scriptsize 2}}
\put(30,0){\circle*{5}{\scriptsize 3}}
\put(30,30){\circle*{5}{\scriptsize 24}}
\put(30,60){\circle*{5}{\scriptsize 124}}
\put(60,0){\circle*{5}{\scriptsize 4}}
\put(60,30){\circle*{5}{\scriptsize 34}}
\put(60,60){\circle*{5}{\scriptsize 134}}
\put(30,0){\line(1,1){30}}
\put(30,0){\line(-1,1){30}}
\put(0,0){\line(1,1){30}}
\put(60,0){\line(-1,1){30}}
\put(0,0){\line(0,1){60}}
\put(30,30){\line(0,1){30}}
\put(60,0){\line(0,1){60}}
\end{picture}
    \caption{Some sets in $\mathbf{B}_4-\{[2],\{1,3\},[4]\}$.}
    \label{case3}
\end{figure}
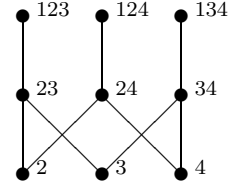
Now assume $c(\varnothing)=2$. We discuss the colors on the above nine sets. Particularly, we consider the possibilities of the single element sets. 
\begin{description}
\item{{\em Subcase 3.1.}} All the three sets $\{2\}$, $\{3\}$, and $\{4\}$ are of color 1.

The three chains $\{\{2,3\},\{1,2,3\}\}$, $\{\{2,4\},\{1,2,4\}\}$, and $\{\{3,4\},\{1,3,4\}\}$ are mutually incomparable. 
At least two of the them contain sets of color 2 only by the assumption of no  $\mathbf{V}_{1,1}$ of color 1. The two chains and $\varnothing$ together form a $\mathbf{V}_{2,2}$ of color 2. 

\item{{\em Subcase 3.2.}} Two of the three sets $\{2\}$, $\{3\}$, and $\{4\}$ are of color 1.

By symmetry, we can assume  $c(\{2\})=c(\{3\})=1$ and $c(\{4\})=2$. If  $c(\{2,3\})=1 $ or $c(\{1,2,3\})=1$, then all the sets in  $\{\{2,4\},\{1,2,4\}\}$ and $\{\{3,4\},\{1,3,4\}\}$ are of color 2. Thus, together with $\varnothing$, they form a $\mathbf{V}_{2,2}$ of color 2. So we conclude that $c(\{2,3\})=c(\{1,2,3\})=2$. Now if any set in $\{\{2,4\},\{1,2,4\}\}\cup\{\{3,4\},\{1,3,4\}\}$ is of color 2, then this set with $\{4\}$, $\{2,3\}$, $\{1,2,3\}$ and $\varnothing$ form a $\mathbf{V}_{2,2}$ of color 2. So all the four sets in $\{\{2,4\},\{1,2,4\}\}\cup\{\{3,4\},\{1,3,4\}\}$ are of color 1. This forces $c(\{2,3,4\})=2$ and $c(\{1\})=2$. 
However, we then have $\{1\}$, $\{1,2,3\}$, $\{4\}$, $\{2,3,4\}$, and $\varnothing$ forming a $\mathbf{V}_{2,2}$ of color 2. 

\item{{\em Subcase 3.3.}} One of the three sets $\{2\}$, $\{3\}$, and $\{4\}$ is of color 1.

First assume $c(\{2\})=1$ and $c(\{3\})=c(\{4\})=2$. 
If there exists a pair of sets, one from 
$\{\{2,3\},\{1,2,3\}\}$ and the other from $\{\{2,4\},\{1,2,4\}\}$, of the same color, then we can find a $\mathbf{V}_{1,1}$ of color 1 or a $\mathbf{V}_{2,2}$ of color 2. Thus, we must have 
$c(\{2,3\})=c(\{1,2,3\})=1$ and  $c(\{2,4\})=c(\{1,2,4\}\})=2$, or the other way around. 
In either cases, we can deduce that $c(\{3,4\})=c(\{1,3,4\})=1$ otherwise we can find four sets of color 2 together with $\varnothing$ form a $\mathbf{V}_{2,2}$. Indeed, if only one of the three sets $\{2\}$, $\{3\}$, and $\{4\}$ is of color 1, then we can always conclude that two of the three chains $\{\{2,3\},\{1,2,3\}\}$, $\{\{2,4\},\{1,2,4\}\}$, and $\{\{3,4\},\{1,3,4\}\}$ contain sets of color 1 and the third one contains sets of color 2 only. This forces $c(\{2,3,4\})=c(\{1\})=2$ since 
\{1\} is contained in any 3-set, $\{2,3,4\}$ contains any 2-set in $\mathcal{F}$, and there is no $\mathbf{V}_{1,1}$ of color 1. 
Eventually, we can pick some $\{x\}$ and $[4]-\{x\}$ from $\mathcal{F}$ together with $\{2,3,4\}$, $\{1\}$, and $\varnothing$ forming a $\mathbf{V}_{2,2}$ of color 2. 

\item{{\em Subcase 3.4.}} All the three sets $\{2\}$, $\{3\}$, and $\{4\}$ are of color 2.

If two of the three chains $\{\{2,3\},\{1,2,3\}\}$, $\{\{2,4\},\{1,2,4\}\}$, and $\{\{3,4\},\{1,3,4\}\}$ contain sets of color 2, then with $\varnothing$ we can easily find a $\mathbf{V}_{2,2}$ of color 2. 
Thus, we assume two of them contain sets of color 1 only. Again, by the reason no $\mathbf{V}_{1,1}$ of color 1, we conclude that $c(\{2,3,4\})=c(\{1\})=2$. Moreover, all 3-sets except for $\{2,3,4\}$ are of color 1 otherwise we have a $\mathbf{V}_{2,2}$ of color 2 as the Subcase 3.3. However, this leads to the situation that we have a $\mathbf{V}_{1,1}$ if $c(\{1,4\})=1$ (with $\{1,2,4\}$ and $\{1,3,4\}$) or 
a $\mathbf{V}_{2,2}$ if $c(\{1,4\})=2$ (with 
$\varnothing$, $\{1\}$,$\{2\}$ or $\{3\}$, $\{2,3,4\}$). So we conclude that a $\mathbf{V}_{2,2}$ of color 2 exists if there is no $\mathbf{V}_{1,1}$ of color 1. 
\end{description}

\begin{comment}
Finally, we give colorings of   $\mathbf{B}_4-\{\varnothing\}$ and 
$\mathbf{B}_{4}-\{S_1,S_2,[4]\}$ for  non-complementary pairs of nonempty sets $S_1$ and $S_2$ that avoid a $\mathbf{V}_{1,1}$ of color 1 nor a $\mathbf{V}_{2,2}$ of color 2. Since every proper subposet of $\mathbf{B}_{4}-\{S,S^c,[4]\}$ is contained in one of the above posets, this will show $\mathbf{B}_{4}-\{S,S^c,[4]\}$ is minimal $(\mathbf{V}_{1,1},\mathbf{V}_{2,2})$-Ramsey. For $\mathbf{B}_4-\{\varnothing\}$, we color each set $X\in\mathbf{B}_4-\{\varnothing\}$ with 1 if $|X|=1$, otherwise with 2. For other posets of the form $\mathbf{B}_{4}-\{S_1,S_2,[4]\}$, the colorings are presented in Figure~\ref{Colorings0}, Figure~\ref{Colorings1}, and Figure~\ref{Colorings2}. The hollow and solid dots refer to color 1 and 2, respectively.
This completes the proof.
\end{comment}

Now we show no other type of minimal $(\mathbf{V}_{1,1},\mathbf{V}_{2,2})$-Ramsey subposet exists.
Consider $\mathbf{B}_4-\{S,[4]\}$. If $S\neq \varnothing$, then it contains some of our previous three types of posets as a  subposet. So they are not minimal. If $S=\varnothing$, we color each set $X\in\mathbf{B}_4-\{\varnothing,[4]\}$ with 1 if $|X|=1$, otherwise with 2. Hence this subposet does not have the Ramsey property.
For the posets of type  $\mathbf{B}_4-\{S_1,S_2,[4]\}$ with $S_1$ and $S_2$ not satisfying our conditions, we give the colorings for each type in  Figure~\ref{Colorings0}, Figure~\ref{Colorings1}, and Figure~\ref{Colorings2}. So they do not have the Ramsey property. 

To complete the proof, we prove that every subposet $\mathbf{Q}$ of $\mathbf{P}$ with $|\mathbf{Q}|=|\mathbf{P}|-1$, where $\mathbf{P}$ is isomorphic to one of the  previous three cases, does not have the Ramsey property. 
In Figure~\ref{32-subsets}, we demonstrate the colorings of the subposets $\mathbf{Q}$ of $\mathbf{P}=\mathbf{B}_4-\{S_1,S_2,[4]\}$ with $|S_1|=|S_2|=2$ obtained by removing one more set of size 2. 
For other subposets $\mathbf{Q}$ of $\mathbf{P}$ with $|\mathbf{Q}|=|\mathbf{P}|-1$  not isomorphic to those shown in Figure~\ref{32-subsets}, they must be isomorphic to the subposets of some posets demonstrated in Figure~\ref{Colorings0}, Figure~\ref{Colorings1}, and Figure~\ref{Colorings2}. Hence they do not have the Ramsey property, and our three cases are minimal. 
\qed

\begin{figure}[ht]
    \centering
\begin{tabular}{ccc}
\begin{picture}(120,80)
\put(40,25){\circle*{4}}
\put(40,10){color 2}
\put(40,55){\circle{4}}
\put(40,40){color 1}
\end{picture}&
\begin{picture}(120,80)
%%
%\put(20,30){\line(3,-2){30}} %\put(40,30){\line(1,-2){10}}
\put(60,30){\line(-1,-2){10}} \put(80,30){\line(-3,-2){30}}
\put(0,50){\line(5,-4){50}}
%%
%\put(20,30){\line(-1,1){20}} 
%\put(20,30){\line(0,1){20}}
%\put(20,30){\line(1,1){20}} \put(40,30){\line(-2,1){40}}
%\put(40,30){\line(1,1){20}} \put(40,30){\line(2,1){40}}
\put(60,30){\line(-2,1){40}} \put(60,30){\line(0,1){20}}
\put(60,30){\line(2,1){40}} \put(80,30){\line(-2,1){40}}
\put(80,30){\line(0,1){20}} \put(80,30){\line(1,1){20}}
\put(20,70){\line(-1,-1){20}} 
\put(20,70){\line(0,-1){20}}
\put(20,70){\line(2,-1){40}} \put(40,70){\line(-2,-1){40}}
\put(40,70){\line(0,-1){20}} \put(40,70){\line(2,-1){40}}
\put(60,70){\line(-2,-1){40}} \put(60,70){\line(-1,-1){20}}
\put(60,70){\line(2,-1){40}} 
\put(80,70){\line(-1,-1){20}}
\put(80,70){\line(0,-1){20}} \put(80,70){\line(1,-1){20}}
\put(50,0){{\scriptsize$ \varnothing$} } 
\put(50,10){\circle*{4} } 
%\put(20,30){\circle*{4}{\scriptsize 1}}
%\put(40,30){\circle*{4}{\scriptsize 2}} 
\put(60,30){\circle*{4}{\scriptsize 3}}
\put(80,30){\circle*{4}{\scriptsize 4}} 
\put(0,50){\circle*{4}{\scriptsize 12}}
\put(20,50){\circle*{4}{\scriptsize 13}} 
\put(40,50){\circle*{4}{\scriptsize 14}}
\put(60,50){\circle*{4}{\scriptsize 23}} 
\put(80,50){\circle*{4}{\scriptsize 24}}
\put(100,50){\circle*{4}{\scriptsize 34}} 
\put(20,70){\circle*{4}{\scriptsize 123}}
\put(40,70){\circle*{4}{\scriptsize 124}} 
\put(60,70){\circle*{4}{\scriptsize 134}}
\put(80,70){\circle*{4}{\scriptsize 234}}
\put(80,30){\textcolor{white}{\circle*{4}}}
\put(100,50){\textcolor{white}{\circle*{4}}}
\put(80,70){\textcolor{white}{\circle*{4}}}
\put(50,10){\textcolor{white}{\circle*{4}}}
%\put(20,70){\textcolor{white}{\circle*{4}}}
%\put(40,70){\textcolor{white}{\circle*{4}}}
\put(80,30){\circle{4}}
\put(100,50){\circle{4}} 
\put(80,70){\circle{4}}
\put(50,10){\circle{4}} 
%\put(20,70){\circle{4}{\scriptsize 123}}
%\put(40,70){\circle{4}{\scriptsize 124}} 
\end{picture}
&
\begin{picture}(120,70)
%%
%\put(20,30){\line(3,-2){30}} 
\put(40,30){\line(1,-2){10}}
\put(60,30){\line(-1,-2){10}} \put(80,30){\line(-3,-2){30}}
%%
%\put(20,30){\line(-1,1){20}} 
%\put(20,30){\line(0,1){20}}
%\put(20,30){\line(1,1){20}} 
\put(40,30){\line(-2,1){40}}
%\put(40,30){\line(1,1){20}} 
\put(40,30){\line(2,1){40}}
\put(60,30){\line(-2,1){40}} %\put(60,30){\line(0,1){20}}
\put(60,30){\line(2,1){40}} \put(80,30){\line(-2,1){40}}
\put(80,30){\line(0,1){20}} \put(80,30){\line(1,1){20}}
\put(20,70){\line(-1,-1){20}} 
\put(20,70){\line(0,-1){20}}
%\put(20,70){\line(2,-1){40}} 
\put(40,70){\line(-2,-1){40}}
\put(40,70){\line(0,-1){20}} \put(40,70){\line(2,-1){40}}
\put(60,70){\line(-2,-1){40}} \put(60,70){\line(-1,-1){20}}
\put(60,70){\line(2,-1){40}} 
%\put(80,70){\line(-1,-1){20}}
\put(80,70){\line(0,-1){20}} \put(80,70){\line(1,-1){20}}
\put(50,0){{\scriptsize$ \varnothing$} } 
\put(50,10){\circle*{4} } 
%\put(20,30){\circle*{4}{\scriptsize 1}}
\put(40,30){\circle*{4}{\scriptsize 2}} 
\put(60,30){\circle*{4}{\scriptsize 3}}
\put(80,30){\circle*{4}{\scriptsize 4}} 
\put(0,50){\circle*{4}{\scriptsize 12}}
\put(20,50){\circle*{4}{\scriptsize 13}} 
\put(40,50){\circle*{4}{\scriptsize 14}}
%\put(60,50){\circle*{4}{\scriptsize 23}} 
\put(80,50){\circle*{4}{\scriptsize 24}}
\put(100,50){\circle*{4}{\scriptsize 34}} 
\put(20,70){\circle*{4}{\scriptsize 123}}
\put(40,70){\circle*{4}{\scriptsize 124}} 
\put(60,70){\circle*{4}{\scriptsize 134}}
\put(80,70){\circle*{4}{\scriptsize 234}}
\put(40,30){\textcolor{white}{\circle*{4}}}
\put(60,30){\textcolor{white}{\circle*{4}}}
\put(0,50){\textcolor{white}{\circle*{4}}}
\put(20,50){\textcolor{white}{\circle*{4}}}
\put(40,70){\textcolor{white}{\circle*{4}}}
\put(60,70){\textcolor{white}{\circle*{4}}}
\put(40,30){\circle{4}}
\put(60,30){\circle{4}} 
\put(0,50){\circle{4}}
\put(20,50){\circle{4}} 
\put(40,70){\circle{4}} 
\put(60,70){\circle{4}} 
\end{picture}
\\
&$\{1\},\{2\}$ & $\{1\},\{2,3\}$ 
\end{tabular}
\caption{Colorings of $\mathbf{B}_4-\{S_1,S_2,[4]\}$ with $|S_1\cap S_2|=0$.}
\label{Colorings0}

\begin{tabular}{ccc}
\begin{picture}(120,90)
%%
%\put(20,30){\line(3,-2){30}} 
\put(40,30){\line(1,-2){10}}
\put(60,30){\line(-1,-2){10}} \put(80,30){\line(-3,-2){30}}
%%
%\put(20,30){\line(-1,1){20}} 
%\put(20,30){\line(0,1){20}}
%\put(20,30){\line(1,1){20}} \put(40,30){\line(-2,1){40}}
\put(40,30){\line(1,1){20}} \put(40,30){\line(2,1){40}}
\put(60,30){\line(-2,1){40}} \put(60,30){\line(0,1){20}}
\put(60,30){\line(2,1){40}} \put(80,30){\line(-2,1){40}}
\put(80,30){\line(0,1){20}} \put(80,30){\line(1,1){20}}
%%
%\put(20,70){\line(-1,-1){20}} 
\put(20,70){\line(0,-1){20}}
\put(20,70){\line(2,-1){40}} %\put(40,70){\line(-2,-1){40}}
\put(40,70){\line(0,-1){20}} \put(40,70){\line(2,-1){40}}
\put(60,70){\line(-2,-1){40}} \put(60,70){\line(-1,-1){20}}
\put(60,70){\line(2,-1){40}} 
\put(80,70){\line(-1,-1){20}}
\put(80,70){\line(0,-1){20}} \put(80,70){\line(1,-1){20}}
\put(50,0){{\scriptsize$ \varnothing$} } 
\put(50,10){\circle*{4} } 
%\put(20,30){\circle*{4}{\scriptsize 1}}
\put(40,30){\circle*{4}{\scriptsize 2}} 
\put(60,30){\circle*{4}{\scriptsize 3}}
\put(80,30){\circle*{4}{\scriptsize 4}} 
%\put(0,50){\circle*{4}{\scriptsize 12}}
\put(20,50){\circle*{4}{\scriptsize 13}} 
\put(40,50){\circle*{4}{\scriptsize 14}}
\put(60,50){\circle*{4}{\scriptsize 23}} 
\put(80,50){\circle*{4}{\scriptsize 24}}
\put(100,50){\circle*{4}{\scriptsize 34}} 
\put(20,70){\circle*{4}{\scriptsize 123}}
\put(40,70){\circle*{4}{\scriptsize 124}} 
\put(60,70){\circle*{4}{\scriptsize 134}}
\put(80,70){\circle*{4}{\scriptsize 234}}
\put(20,50){\textcolor{white}{\circle*{4}}}
\put(40,50){\textcolor{white}{\circle*{4}}}
\put(60,50){\textcolor{white}{\circle*{4}}}
\put(80,50){\textcolor{white}{\circle*{4}}}
\put(100,50){\textcolor{white}{\circle*{4}}}
\put(20,70){\textcolor{white}{\circle*{4}}}
\put(40,70){\textcolor{white}{\circle*{4}}}
\put(20,50){\circle{4}}
\put(40,50){\circle{4}}
\put(60,50){\circle{4}} 
\put(80,50){\circle{4}}
\put(100,50){\circle{4}} 
\put(20,70){\circle{4}}
\put(40,70){\circle{4}} 
\end{picture}
&
\begin{picture}(120,70)
%%
%\put(20,30){\line(3,-2){30}} 
\put(40,30){\line(1,-2){10}}
\put(60,30){\line(-1,-2){10}} \put(80,30){\line(-3,-2){30}}
%%
%\put(20,30){\line(-1,1){20}} 
%\put(20,30){\line(0,1){20}}
%\put(20,30){\line(1,1){20}} 
\put(40,30){\line(-2,1){40}}
\put(40,30){\line(1,1){20}} \put(40,30){\line(2,1){40}}
\put(60,30){\line(-2,1){40}} \put(60,30){\line(0,1){20}}
\put(60,30){\line(2,1){40}} \put(80,30){\line(-2,1){40}}
\put(80,30){\line(0,1){20}} \put(80,30){\line(1,1){20}}
%%
%\put(20,70){\line(-1,-1){20}} 
%\put(20,70){\line(0,-1){20}}
%\put(20,70){\line(2,-1){40}} 
\put(40,70){\line(-2,-1){40}}
\put(40,70){\line(0,-1){20}} \put(40,70){\line(2,-1){40}}
\put(60,70){\line(-2,-1){40}} \put(60,70){\line(-1,-1){20}}
\put(60,70){\line(2,-1){40}} 
\put(80,70){\line(-1,-1){20}}
\put(80,70){\line(0,-1){20}} \put(80,70){\line(1,-1){20}}
\put(50,0){{\scriptsize$ \varnothing$} } 
\put(50,10){\circle*{4} } 
%\put(20,30){\circle*{4}{\scriptsize 1}}
\put(40,30){\circle*{4}{\scriptsize 2}} 
\put(60,30){\circle*{4}{\scriptsize 3}}
\put(80,30){\circle*{4}{\scriptsize 4}} 
\put(0,50){\circle*{4}{\scriptsize 12}}
\put(20,50){\circle*{4}{\scriptsize 13}} 
\put(40,50){\circle*{4}{\scriptsize 14}}
\put(60,50){\circle*{4}{\scriptsize 23}} 
\put(80,50){\circle*{4}{\scriptsize 24}}
\put(100,50){\circle*{4}{\scriptsize 34}} 
%\put(20,70){\circle*{4}{\scriptsize 123}}
\put(40,70){\circle*{4}{\scriptsize 124}} 
\put(60,70){\circle*{4}{\scriptsize 134}}
\put(80,70){\circle*{4}{\scriptsize 234}}
%\put(0,50){\textcolor{white}{\circle*{4}}}
%\put(20,50){\textcolor{white}{\circle*{4}}}
%\put(40,50){\textcolor{white}{\circle*{4}}}
%\put(60,50){\textcolor{white}{\circle*{4}}}
%\put(80,50){\textcolor{white}{\circle*{4}}}

\put(80,70){\textcolor{white}{\circle*{4}}}
\put(100,50){\textcolor{white}{\circle*{4}}}
\put(80,30){\textcolor{white}{\circle*{4}}}
\put(50,10){\textcolor{white}{\circle*{4}}}
\put(80,70){\circle{4}} 
\put(100,50){\circle{4}} 
\put(80,30){\circle{4}} 
\put(50,10){\circle{4}} 

\end{picture}
&
\begin{picture}(120,70)
\put(20,30){\line(3,-2){30}} \put(40,30){\line(1,-2){10}}
\put(60,30){\line(-1,-2){10}} \put(80,30){\line(-3,-2){30}}
%%
%\put(20,30){\line(-1,1){20}} 
\put(20,30){\line(0,1){20}}
\put(20,30){\line(1,1){20}} %\put(40,30){\line(-2,1){40}}
\put(40,30){\line(1,1){20}} \put(40,30){\line(2,1){40}}
\put(60,30){\line(-2,1){40}} \put(60,30){\line(0,1){20}}
\put(60,30){\line(2,1){40}} \put(80,30){\line(-2,1){40}}
\put(80,30){\line(0,1){20}} \put(80,30){\line(1,1){20}}
%%
%\put(20,70){\line(-1,-1){20}} 
\put(20,70){\line(0,-1){20}}
\put(20,70){\line(2,-1){40}} %\put(40,70){\line(-2,-1){40}}
\put(40,70){\line(0,-1){20}} \put(40,70){\line(2,-1){40}}
\put(60,70){\line(-2,-1){40}} \put(60,70){\line(-1,-1){20}}
\put(60,70){\line(2,-1){40}} 
%\put(80,70){\line(-1,-1){20}}
%\put(80,70){\line(0,-1){20}} %\put(80,70){\line(1,-1){20}}
%%
\put(50,0){{\scriptsize$ \varnothing$} } 
\put(50,10){\circle*{4} } 
\put(20,30){\circle*{4}{\scriptsize 1}}
\put(40,30){\circle*{4}{\scriptsize 2}} 
\put(60,30){\circle*{4}{\scriptsize 3}}
\put(80,30){\circle*{4}{\scriptsize 4}} 
%\put(0,50){\circle*{4}{\scriptsize 12}}
\put(20,50){\circle*{4}{\scriptsize 13}} 
\put(40,50){\circle*{4}{\scriptsize 14}}
\put(60,50){\circle*{4}{\scriptsize 23}} 
\put(80,50){\circle*{4}{\scriptsize 24}}
\put(100,50){\circle*{4}{\scriptsize 34}} 
\put(20,70){\circle*{4}{\scriptsize 123}}
\put(40,70){\circle*{4}{\scriptsize 124}} 
\put(60,70){\circle*{4}{\scriptsize 134}}
%\put(80,70){\circle*{4}{\scriptsize 234}}
\put(20,50){\textcolor{white}{\circle*{4}}}
\put(40,50){\textcolor{white}{\circle*{4}}}
\put(60,50){\textcolor{white}{\circle*{4}}}
\put(80,50){\textcolor{white}{\circle*{4}}}
\put(100,50){\textcolor{white}{\circle*{4}}}
\put(20,70){\textcolor{white}{\circle*{4}}}
\put(40,70){\textcolor{white}{\circle*{4}}}
\put(20,50){\circle{4}}
\put(40,50){\circle{4}}
\put(60,50){\circle{4}} 
\put(80,50){\circle{4}}
\put(100,50){\circle{4}} 
\put(20,70){\circle{4}}
\put(40,70){\circle{4}} 
\end{picture}
\\
$\{1\},\{1,2\}$ & $\{1\},\{1,2,3\}$ & $\{1,2\},\{2,3,4\}$ 
\end{tabular}
    \caption{Colorings of $\mathbf{B}_4-\{S_1,S_2,[4]\}$ with $|S_1\cap S_2|=1$.}
    \label{Colorings1}

\begin{tabular}{cc}
\begin{picture}(120,90)
\put(20,30){\line(3,-2){30}} \put(40,30){\line(1,-2){10}}
\put(60,30){\line(-1,-2){10}} \put(80,30){\line(-3,-2){30}}
%%
%\put(20,30){\line(-1,1){20}} 
\put(20,30){\line(0,1){20}}
\put(20,30){\line(1,1){20}} %\put(40,30){\line(-2,1){40}}
\put(40,30){\line(1,1){20}} \put(40,30){\line(2,1){40}}
\put(60,30){\line(-2,1){40}} \put(60,30){\line(0,1){20}}
\put(60,30){\line(2,1){40}} \put(80,30){\line(-2,1){40}}
\put(80,30){\line(0,1){20}} \put(80,30){\line(1,1){20}}
%%
%\put(20,70){\line(-1,-1){20}} 
%\put(20,70){\line(0,-1){20}}
%\put(20,70){\line(2,-1){40}} %\put(40,70){\line(-2,-1){40}}
\put(40,70){\line(0,-1){20}} \put(40,70){\line(2,-1){40}}
\put(60,70){\line(-2,-1){40}} \put(60,70){\line(-1,-1){20}}
\put(60,70){\line(2,-1){40}} 
\put(80,70){\line(-1,-1){20}}
\put(80,70){\line(0,-1){20}} \put(80,70){\line(1,-1){20}}
\put(50,0){{\scriptsize$ \varnothing$} } 
\put(50,10){\circle*{4} } 
\put(20,30){\circle*{4}{\scriptsize 1}}
\put(40,30){\circle*{4}{\scriptsize 2}} 
\put(60,30){\circle*{4}{\scriptsize 3}}
\put(80,30){\circle*{4}{\scriptsize 4}} 
%\put(0,50){\circle*{4}{\scriptsize 12}}
\put(20,50){\circle*{4}{\scriptsize 13}} 
\put(40,50){\circle*{4}{\scriptsize 14}}
\put(60,50){\circle*{4}{\scriptsize 23}} 
\put(80,50){\circle*{4}{\scriptsize 24}}
\put(100,50){\circle*{4}{\scriptsize 34}} 
%\put(20,70){\circle*{4}{\scriptsize 123}}
\put(40,70){\circle*{4}{\scriptsize 124}} 
\put(60,70){\circle*{4}{\scriptsize 134}}
\put(80,70){\circle*{4}{\scriptsize 234}}
\put(50,10){\textcolor{white}{\circle*{4}}}
\put(80,30){\textcolor{white}{\circle*{4}}}
\put(100,50){\textcolor{white}{\circle*{4}}}
\put(80,70){\textcolor{white}{\circle*{4}}}
\put(50,10){\circle{4}}
\put(80,30){\circle{4}}
\put(80,50){\circle{4}}
\put(100,50){\circle{4}} 
\put(80,70){\circle{4}} 
\end{picture}
& 
\begin{picture}(120,70)
\put(20,30){\line(3,-2){30}} \put(40,30){\line(1,-2){10}}
\put(60,30){\line(-1,-2){10}} \put(80,30){\line(-3,-2){30}}
\put(20,30){\line(-1,1){20}} 
\put(20,30){\line(0,1){20}}
\put(20,30){\line(1,1){20}} \put(40,30){\line(-2,1){40}}
\put(40,30){\line(1,1){20}} \put(40,30){\line(2,1){40}}
\put(60,30){\line(-2,1){40}} \put(60,30){\line(0,1){20}}
\put(60,30){\line(2,1){40}} \put(80,30){\line(-2,1){40}}
\put(80,30){\line(0,1){20}} \put(80,30){\line(1,1){20}}
%%
%\put(20,70){\line(-1,-1){20}} 
%\put(20,70){\line(0,-1){20}}
%\put(20,70){\line(2,-1){40}} 
\put(40,70){\line(-2,-1){40}}
\put(40,70){\line(0,-1){20}} \put(40,70){\line(2,-1){40}}
\put(60,70){\line(-2,-1){40}} \put(60,70){\line(-1,-1){20}}
\put(60,70){\line(2,-1){40}} 
%\put(80,70){\line(-1,-1){20}}
%\put(80,70){\line(0,-1){20}} %\put(80,70){\line(1,-1){20}}
%%
\put(50,0){{\scriptsize$ \varnothing$} } 
\put(50,10){\circle*{4} } 
\put(20,30){\circle*{4}{\scriptsize 1}}
\put(40,30){\circle*{4}{\scriptsize 2}} 
\put(60,30){\circle*{4}{\scriptsize 3}}
\put(80,30){\circle*{4}{\scriptsize 4}} 
\put(0,50){\circle*{4}{\scriptsize 12}}
\put(20,50){\circle*{4}{\scriptsize 13}} 
\put(40,50){\circle*{4}{\scriptsize 14}}
\put(60,50){\circle*{4}{\scriptsize 23}} 
\put(80,50){\circle*{4}{\scriptsize 24}}
\put(100,50){\circle*{4}{\scriptsize 34}} 
%\put(20,70){\circle*{4}{\scriptsize 123}}
\put(40,70){\circle*{4}{\scriptsize 124}} 
\put(60,70){\circle*{4}{\scriptsize 134}}
%\put(80,70){\circle*{4}{\scriptsize 234}}
\put(50,10){\textcolor{white}{\circle*{4}}}
\put(80,30){\textcolor{white}{\circle*{4}}}
\put(100,50){\textcolor{white}{\circle*{4}}}
\put(60,70){\textcolor{white}{\circle*{4}}}
\put(50,10){\circle{4}}
\put(80,30){\circle{4}}
\put(100,50){\circle{4}}
\put(60,70){\circle{4}} 
\end{picture}
  \\
$\{1,2\},\{1,2,3\}$    &$\{1,2,3\},\{2,3,4\}$
\end{tabular}
    \caption{Colorings of $\mathbf{B}_4-\{S_1,S_2,[4]\}$ with $|S_1\cap S_2|=2$.}
    \label{Colorings2}
\end{figure}

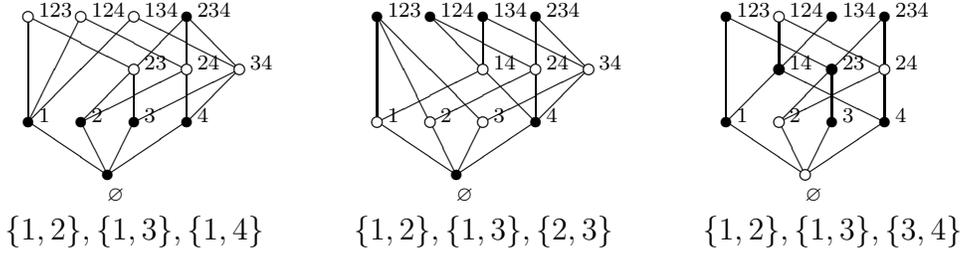
\begin{figure}[ht]
    \centering
\begin{tabular}{ccc}
\begin{picture}(120,80)
\put(20,30){\line(3,-2){30}} 
\put(40,30){\line(1,-2){10}}
\put(60,30){\line(-1,-2){10}} \put(80,30){\line(-3,-2){30}}
\put(20,30){\line(1,2){20}} 
\put(20,30){\line(0,1){40}}
\put(20,30){\line(1,1){40}} 
%\put(40,30){\line(-2,1){40}}
\put(40,30){\line(1,1){20}} \put(40,30){\line(2,1){40}}
%\put(60,30){\line(-2,1){40}} 
\put(60,30){\line(0,1){20}}
\put(60,30){\line(2,1){40}} %\put(80,30){\line(-2,1){40}}
\put(80,30){\line(0,1){20}} \put(80,30){\line(1,1){20}}
%%
%\put(20,70){\line(-1,-1){20}} 
%\put(20,70){\line(0,-1){20}}
\put(20,70){\line(2,-1){40}} 
%\put(40,70){\line(-2,-1){40}}
%\put(40,70){\line(0,-1){20}} 
\put(40,70){\line(2,-1){40}}
%\put(60,70){\line(-2,-1){40}} %\put(60,70){\line(-1,-1){20}}
\put(60,70){\line(2,-1){40}} 
\put(80,70){\line(-1,-1){20}}
\put(80,70){\line(0,-1){20}} \put(80,70){\line(1,-1){20}}
\put(50,0){{\scriptsize$ \varnothing$} } 
\put(50,10){\circle*{4} } 
\put(20,30){\circle*{4}{\scriptsize 1}}
\put(40,30){\circle*{4}{\scriptsize 2}} 
\put(60,30){\circle*{4}{\scriptsize 3}}
\put(80,30){\circle*{4}{\scriptsize 4}} 
%\put(0,50){\circle*{4}{\scriptsize 12}}
%\put(20,50){\circle*{4}{\scriptsize 13}} 
%\put(40,50){\circle*{4}{\scriptsize 14}}
\put(60,50){\circle*{4}{\scriptsize 23}} 
\put(80,50){\circle*{4}{\scriptsize 24}}
\put(100,50){\circle*{4}{\scriptsize 34}} 
\put(20,70){\circle*{4}{\scriptsize 123}}
\put(40,70){\circle*{4}{\scriptsize 124}} 
\put(60,70){\circle*{4}{\scriptsize 134}}
\put(80,70){\circle*{4}{\scriptsize 234}}
%\put(20,50){\textcolor{white}{\circle*{4}}}
\put(60,70){\textcolor{white}{\circle*{4}}}
\put(60,50){\textcolor{white}{\circle*{4}}}
\put(80,50){\textcolor{white}{\circle*{4}}}
\put(100,50){\textcolor{white}{\circle*{4}}}
\put(20,70){\textcolor{white}{\circle*{4}}}
\put(40,70){\textcolor{white}{\circle*{4}}}
%\put(20,50){\circle{4}}
\put(60,70){\circle{4}}
\put(60,50){\circle{4}} 
\put(80,50){\circle{4}}
\put(100,50){\circle{4}} 
\put(20,70){\circle{4}}
\put(40,70){\circle{4}} 
\end{picture}
&
\begin{picture}(120,70)
\put(20,30){\line(3,-2){30}} 
\put(40,30){\line(1,-2){10}}
\put(60,30){\line(-1,-2){10}} \put(80,30){\line(-3,-2){30}}
%%
%\put(20,30){\line(1,2){20}} 
\put(20,30){\line(0,1){40}}
%\put(20,30){\line(1,1){40}} 
%\put(40,30){\line(-2,1){40}}
\put(20,30){\line(2,1){40}} \put(40,30){\line(2,1){40}}
%\put(60,30){\line(-2,1){40}} 
\put(80,30){\line(-1,1){20}}
\put(60,30){\line(2,1){40}} %\put(80,30){\line(-2,1){40}}
\put(80,30){\line(0,1){20}} \put(80,30){\line(1,1){20}}
%%
%\put(20,70){\line(-1,-1){20}} 
%\put(20,70){\line(0,-1){20}}
\put(20,70){\line(1,-1){40}} 
\put(20,70){\line(1,-2){20}} 
%\put(40,70){\line(-2,-1){40}}
%\put(40,70){\line(0,-1){20}} 
\put(40,70){\line(1,-1){20}}
%\put(60,70){\line(-2,-1){40}} 
\put(60,70){\line(0,-1){20}}
\put(60,70){\line(2,-1){40}} 
\put(80,70){\line(1,-1){20}}
\put(80,70){\line(0,-1){20}} \put(40,70){\line(2,-1){40}}
\put(50,0){{\scriptsize$ \varnothing$} } 
\put(50,10){\circle*{4} } 
\put(20,30){\circle*{4}{\scriptsize 1}}
\put(40,30){\circle*{4}{\scriptsize 2}} 
\put(60,30){\circle*{4}{\scriptsize 3}}
\put(80,30){\circle*{4}{\scriptsize 4}} 
%\put(0,50){\circle*{4}{\scriptsize 12}}
%\put(20,50){\circle*{4}{\scriptsize 13}} 
%\put(40,50){\circle*{4}{\scriptsize 14}}
\put(60,50){\circle*{4}{\scriptsize 14}} 
\put(80,50){\circle*{4}{\scriptsize 24}}
\put(100,50){\circle*{4}{\scriptsize 34}} 
\put(20,70){\circle*{4}{\scriptsize 123}}
\put(40,70){\circle*{4}{\scriptsize 124}} 
\put(60,70){\circle*{4}{\scriptsize 134}}
\put(80,70){\circle*{4}{\scriptsize 234}}
%\put(20,50){\textcolor{white}{\circle*{4}}}
\put(60,50){\textcolor{white}{\circle*{4}}}
\put(80,50){\textcolor{white}{\circle*{4}}}
\put(100,50){\textcolor{white}{\circle*{4}}}
\put(20,30){\textcolor{white}{\circle*{4}}}
\put(60,30){\textcolor{white}{\circle*{4}}}
\put(40,30){\textcolor{white}{\circle*{4}}}
%\put(20,50){\circle{4}}
\put(60,30){\circle{4}}
\put(60,50){\circle{4}} 
\put(80,50){\circle{4}}
\put(100,50){\circle{4}} 
\put(20,30){\circle{4}}
\put(40,30){\circle{4}} 
\end{picture}
&
\begin{picture}(120,70)
\put(20,30){\line(3,-2){30}} \put(40,30){\line(1,-2){10}}
\put(60,30){\line(-1,-2){10}} \put(80,30){\line(-3,-2){30}}
%%
%\put(20,30){\line(-1,1){20}} 
\put(20,30){\line(0,1){40}}
\put(20,30){\line(1,1){20}} 
%\put(40,30){\line(-2,1){40}}
\put(40,30){\line(1,1){20}} \put(40,30){\line(2,1){40}}
%\put(60,30){\line(-2,1){40}} 
\put(60,30){\line(0,1){20}}
%\put(60,30){\line(2,1){40}} 
\put(80,30){\line(-2,1){40}}
\put(80,30){\line(0,1){20}} %\put(80,30){\line(1,1){20}}
%%
%\put(20,70){\line(-1,-1){20}} 
%\put(20,70){\line(0,-1){20}}
\put(20,70){\line(2,-1){40}} %\put(40,70){\line(-2,-1){40}}
\put(40,70){\line(0,-1){20}} \put(40,70){\line(2,-1){40}}
%\put(60,70){\line(-2,-1){40}} 
\put(60,70){\line(-1,-1){20}}
%\put(60,70){\line(2,-1){40}} 
\put(80,70){\line(-1,-1){20}}
\put(80,70){\line(0,-1){20}} %\put(80,70){\line(1,-1){20}}
\put(50,0){{\scriptsize$ \varnothing$} } 
\put(50,10){\circle*{4} } 
\put(20,30){\circle*{4}{\scriptsize 1}}
\put(40,30){\circle*{4}{\scriptsize 2}} 
\put(60,30){\circle*{4}{\scriptsize 3}}
\put(80,30){\circle*{4}{\scriptsize 4}} 
%\put(0,50){\circle*{4}{\scriptsize 12}}
%\put(20,50){\circle*{4}{\scriptsize 13}} 
\put(40,50){\circle*{4}{\scriptsize 14}}
\put(60,50){\circle*{4}{\scriptsize 23}} 
\put(80,50){\circle*{4}{\scriptsize 24}}
%\put(100,50){\circle*{4}{\scriptsize 34}} 
\put(20,70){\circle*{4}{\scriptsize 123}}
\put(40,70){\circle*{4}{\scriptsize 124}} 
\put(60,70){\circle*{4}{\scriptsize 134}}
\put(80,70){\circle*{4}{\scriptsize 234}}
\put(40,30){\textcolor{white}{\circle*{4}}}
\put(50,10){\textcolor{white}{\circle*{4}}}
\put(80,50){\textcolor{white}{\circle*{4}}}
\put(100,50){\textcolor{white}{\circle*{4}}}
%\put(20,70){\textcolor{white}{\circle*{4}}}
\put(40,70){\textcolor{white}{\circle*{4}}}
\put(40,30){\circle{4}}
\put(40,50){\circle{4}}
\put(60,50){\circle{4}} 
\put(80,50){\circle{4}}
\put(50,10){\circle{4}} 
%\put(20,70){\circle{4}}
\put(40,70){\circle{4}} 
\end{picture}
\\
$\{1,2\},\{1,3\},\{1,4\}$ & $\{1,2\},\{1,3\},\{2,3\}$ & $\{1,2\},\{1,3\},\{3,4\}$ 
\end{tabular}
\caption{Colorings of $\mathbf{B}_4-\{S_1,S_2,S_3,[4]\}$ with $|S_1|=|S_2|=|S_3|=2$.}
    \label{32-subsets}
\end{figure}

\section{Boolean Rainbow Ramsey Numbers}

In this section, we study the Boolean rainbow Ramsey numbers. In addition to the proof of  Theorem~\ref{RRJA}, we will present an improvement of Theorem~\ref{RRPA} with $\mathbf{A}_2$. From the first part of Theorem~\ref{RRPA}, we have 
\[ RR(\mathbf{P},\mathbf{A}_2)\le \lfloor\lambda^*_{\max}(\mathbf{P})\rfloor+2.
\]
However, even for a small poset, like $\mathbf{B}_2$, the exact value of $\lambda^*_{\max}(\mathbf{P})$ is unknown.
Nevertheless, we manage to determine $RR(\mathbf{P},\mathbf{A}_2)$  in terms of the 2-dimension of $\mathbf{P}$ and the number of extremal elements in $\mathbf{P}$.
Let the {\em minimum element} and the {\em maximum  element} of a poset $\mathbf{P}$ be the element $x$ such that $x \le z$ and
the element $y$ such that $z \le y$ for all $z \in \mathbf{P}$, respectively. The two extremal elements do not necessarily exist in every poset.
Let $m(\mathbf{P})$ be the number of the extremal elements in the poset $\mathbf{P}$. It is obvious $m(\mathbf{P})\in\{0,1,2\}$.
The next proposition determines $RR(\mathbf{P}, \mathbf{A}_2)$.

\begin{proposition}\label{RRPA2}
For any poset $\mathbf{P}$ on at least 2 elements,
\[
RR(\mathbf{P}, \mathbf{A}_2) = \dim_2 (\mathbf{P}) + m(\mathbf{P}).\]
\end{proposition}

\begin{proof}
Let $n = \dim_2 (\mathbf{P}) + m(\mathbf{P})$.
For the lower bound, define a coloring $c$ of $\mathbf{B}_{n-1}$ by
\[
c(X)=\left\{
\begin{array}{ll}
1,& X=\varnothing,\\
2,& X=[n-1],\\
3,& \mbox{otherwise}.
\end{array}
\right.
\]
Since no set other than $\varnothing$ and $[n-1]$ is colored with 1 or 2,
there is no monochromatic $\mathbf{P}$ of color 1 or 2.
Let us see the elements of color 3.
If $m(\mathbf{P}) = 0$, then $n=\dim_2 (\mathbf{P})$.  
There is no poset $\mathbf{P}$ contained in $\mathbf{B}_{n-1} $ by the definition of $\dim_2 (\mathbf{P})$.
If $m(\mathbf{P})=1$, then $n-1 = \dim_2 (\mathbf{P})$.
We may assume that $\mathbf{P}$ is a poset which has the minimum but not the maximum element.
Suppose there is a $\mathbf{P}$ of color 3 contained in $\mathbf{B}_{n-1}$ with some $S\in \mathbf{B}_{n-1}$ as the minimum element of $\mathbf{P}$.
Since $S \neq\varnothing$, we may assume $n-1 \in S$. Then $n-1 \in X$ for all $X \in \mathbf{P}$.
Let $\mathbf{P}' = \{ X-\{n-1\}\mid X \in\mathbf{P} \}$. Then $\mathbf{P}'$ is isomorphic to $\mathbf{P}$, and every element in $\mathbf{P}'$ is a subset of $[n-2]$.
This implies $\mathbf{P}$ is contained in $\mathbf{B}_{n-2}$, which contradicts the assumption $\dim(\mathbf{P})=n-1$.
The case of $\mathbf{P}$ containing a maximum element is analogous. For $m(\mathbf{P})= 2$, again suppose that there is a $\mathbf{P}$ of color 3. Let $Y$ and $Z$ be the maximum and minimum elements in the monochromatic $\mathbf{P}$, respectively.
Since $\varnothing\neq Z\subset Y\neq [n-1] $,
$\mathbf{P}$ is contained in $\mathbf{B}_{n-1}^{i,j}$ for some $i,j$.
However, $\mathbf{B}_{n-1}^{i,j}$ is isomorphic to $\mathbf{B}_{n-3}$ and $n-3=\dim_2 (\mathbf{P})-1$. This is  a contradiction.
As a consequence, there is no $\mathbf{P}$ of color 3. On the other hand, any two sets of distinct colors are comparable. So, there is no rainbow $\mathbf{A}_2$ contained in $\mathbf{B}_{n-1}$ as well.

For the upper bound, we assume that $c'$ is a coloring of $\mathbf{B}_n$ and does not contain a rainbow $\mathbf{A}_2$. Then we prove there is a monochromatic $\mathbf{P}$. 
To avoid a rainbow $\mathbf{A}_2$, 
we cannot have a set in $\mathbf{B}_{n}^{i,j}$ and a set in $\mathbf{B}_{n}^{j,i}$ colored differently for all $i\neq j\in[n]$.
This unifies the colors of all nontrivial subsets of $[n]$.
If $m(\mathbf{P}) = 0$, then $n =\dim_2 (\mathbf{P})$ and $\mathbf{B}_n$ contains $\mathbf{P}$ as a subposet.
Since $\mathbf{P}$ does not have the extremal elements, all sets forming $\mathbf{P}$ are nontrivial and have the same color.
We have a monochromatic $\mathbf{P}$ contained in $\mathbf{B}_n$.
For $m(\mathbf{P}) = 1$, we may assume that $\mathbf{P}$ has the minimum element. As $n>\dim_2(\mathbf{P})$, $\mathbf{B}_n -\{ \varnothing\}$ contains $\mathbf{B}_{\dim_2 (\mathbf{P})}$ and thus $\mathbf{P}$ as a subposet.
Since $\mathbf{P}$ does not have the maximum element, the poset $\mathbf{P}$ does not contain $[n]$ as its element. So it is monochromatic.
For $m(\mathbf{P})=2$, again we have that  $\mathbf{B}_n-\{\varnothing,[n]\}$ contains $\mathbf{B}_{\dim(\mathbf{P})}$ and hence  $\mathbf{P}$ as a subposet. 

So we obtain a monochromatic $\mathbf{P}$ as desired. Conclusively, $RR(\mathbf{P}, \mathbf{A}_2) = \dim_2 (\mathbf{P}) +m(\mathbf{P})$.
\end{proof}

\noindent{\em Proof of Theorem~\ref{RRJA}.}
We first demonstrate a coloring method in~\cite{CCLL} that provides a lower bound of $RR(\mathbf{V}_{m,n}, \mathbf{A}_k)$.
For any $X\in\mathbf{B}_{n(k-1)+1}$, color $X$ with $\lceil \frac{|X|}{n} \rceil$ if $X\not\in \{\varnothing,[n(k-1)+1]\}$, and with $k$ otherwise. 
Observe that there are only $k$ colors for this coloring and every monochromatic chain contains at most $n$ elements. A family of $k$ subsets with $k$ distinct colors must contain either $\varnothing$ or $[n(k-1)+1]$, 
so it cannot be an antichain.
On the other hand, since $\mathbf{V}_{m,n}$ contains a chain on $n+1$ elements, there is no monochromatic $\mathbf{V}_{m,n}$ as well.
Thus $RR(\mathbf{V}_{m,n}, \mathbf{A}_k) > n(k-1)+1$.

Next we prove $RR(\mathbf{V}_{m,n}, \mathbf{A}_k) \le n(k-1)+2$ by induction on $k$.
For $k=2$, $RR(\mathbf{V}_{m,n}, \mathbf{A}_2) = \dim_2 (\mathbf{V}_{m,n}) + m(\mathbf{V}_{m,n})$ by Proposition~\ref{RRPA2}.
Since $\dim_2 (\mathbf{V}_{m,n})= n+1$ and $m(\mathbf{V}_{m,n}) = 1$, 
$RR(\mathbf{V}_{m,n}, \mathbf{A}_{2}) = n+2$ satisfies the inequality. Suppose the inequality holds for some $k-1 \ge 2$, that is $RR(\mathbf{V}_{m,n}, \mathbf{A}_{k-1}) \le n(k-2)+2$.
Now consider any coloring $c$ of $\mathbf{B}_{n(k-1)+2}$.
Assume that $\mathbf{B}_{n(k-1)+2}$ does not contain a monochromatic $\mathbf{V}_{m,n}$ for the coloring $c$.
We are going to show there is a rainbow $\mathbf{A}_{k}$.

{\em Claim:} There exists some $\mathbf{B}_{n(k-1)+2}^{i,\ell}$ with at least $k$ colors on its elements.

Observe that for any given $i\in [n(k-1)+2]$, the family 
\[\mathcal{F}=\{X \subseteq [n(k-1)+2] \mid i\in X \}\]
is isomorphic to $\mathbf{B}_{n(k-1)+1}$.
Therefore, either there are at least $k$ colors appearing in $\mathcal{F}$, or there is a monochromatic $\mathbf{V}_{m,n}$ by Theorem~\ref{IT}. We have the former case by our assumption of $c$.
Moreover, suppose there are exactly $k$ colors $1,\ldots,k$ on the subsets in $\mathcal{F}$, and $[n(k-1)+2]$ is the only subset which receives color $k$ under $c$. We may replace it with any color less than $k$ to get a $(k-1)$-coloring of $\mathbf{B}_{n(k-1)+1}$. 
By Theorem~\ref{IT} again, $\mathcal{F}$ contains a monochromatic $\mathbf{V}_{m,n}$ as a subposet. 
Since $[n(k-1)+2]$ cannot be an element in $\mathbf{V}_{m,n}$, this monochromatic $\mathbf{V}_{m,n}$ in $\mathcal{F}$ already exists in the original coloring $c$.
This contradicts our assumption.
As a conclusion, there are at least $k$ colors on the subsets in $\mathcal{F}-\{[n(k-1)+2]\}$.
Fix some $j \neq i$. 
Suppose there are at most $k-1$ colors on $\mathbf{B}_{n(k-1)+2}^{i,j}$.
Pick $X \in \mathcal{F} -\{[n(k-1)+2]\}$ with $c(X)$ not appearing on the elements in  $\mathbf{B}_{n(k-1)+2}^{i,j}$.
Then $X$ belongs to  $\mathbf{B}_{n(k-1)+2}^{i,\ell}$ for some $\ell \ne i,j$.
For each subset $W \in \mathbf{B}_{n(k-1)+2}^{i,j} \cap \mathbf{B}_{n(k-1)+2}^{i,\ell}$, we have $i \in W$, $ j \not\in W$, and $\ell \not\in W$.
More precisely, 
\[
\mathbf{B}_{n(k-1)+2}^{i,j} \cap \mathbf{B}_{n(k-1)+2}^{i,\ell} = \{W \mid \{i\} \subset W \subset [n(k-1)+2] -\{j,\ell\} \},
\] and it is isomorphic to $\mathbf{B}_{n(k-1)-1}$.
Since $n(k-1)-1\ge n(k-2)+1$ for $n\ge 2$, by Theorem \ref{IT} and the assumption of $c$, there are at least $k-1$ colors in $\mathbf{B}_{n(k-1)+2}^{i,j} \cap \mathbf{B}_{n(k-1)+2}^{i,\ell}$.
Thus, the colors in $\mathbf{B}_{n(k-1)+2}^{i,j} \cap \mathbf{B}_{n(k-1)+2}^{i,\ell}$ are the same as those in $\mathbf{B}_{n(k-1)+2}^{i,j}$.
On the other hand,  $\mathbf{B}_{n(k-1)+2}^{i,\ell}$ contains $\mathbf{B}_{n(k-1)+2}^{i,j} \cap \mathbf{B}_{n(k-1)+2}^{i,\ell}$ and 
$X$. So there are $k$ colors in $\mathbf{B}_{n(k-1)+2}^{i,\ell}$, and the claim is proved.

Finally, since $n(k-1)\ge n(k-2)+2$ for $n\ge 2$, by inductive hypothesis, we have that $\mathbf{B}_{n(k-1)+2}^{\ell,i}$ contains $\mathbf{B}_{n(k-2)+2}$ and hence a rainbow $\mathbf{A}_{k-1}$.
We now pick an element in $\mathbf{B}_{n(k-1)+2}^{i,\ell}$ whose color does not appear in the former rainbow $\mathbf{A}_{k-1}$. This element together with the  rainbow $\mathbf{A}_{k-1}$ form a rainbow $\mathbf{A}_{k}$ in $\mathbf{B}_{n(k-1)+2}$.
\qed

\section{Concluding Remarks and Open Problems}

The posets $\mathbf{P}$ with $\dim_2(\mathbf{P})=2$ are $\mathbf{C}_3$ the chain on three elements,  $\mathbf{A}_2$, $\mathbf{V}_{1,1}$ and its dual, and the Boolean lattice $\mathbf{B}_2$.
The only unknown $R_k(\mathbf{P})$ is when $\mathbf{P}=\mathbf{B}_2$. Axenovich and Walzer~\cite{AW} showed $R_2(\mathbf{B}_2)=4$. Indeed, we can show that $R_3(\mathbf{B}_2)=6$ using the facts of the Boolean rainbow Ramsey numbers. However, the argument we used to derive it is relatively complicated and cannot be generalized further. So we only publish it on arXiv. Interested readers can see arXiv:1909.11370 for the details.
Anyway, it is nature to investigate $R_k(\mathbf{B}_2)$ after we have settled $R_k(\mathbf{V}_{1,1})$.
Hopefully, the result of $R_k(\mathbf{V}_{1,1})$ could do some help to this problem.

Recall that we introduced the weak subposets and mentioned a reference~\cite{CS} of the Boolean Ramsey number for the weak subposets.   
Grosz {\em et al.}~\cite{GMT} used  $R_w(\mathbf{P}_1,\mathbf{P}_2\ldots, \mathbf{P}_k)$ to denote the weak version of Boolean Ramsey number. It is clear that 
\begin{equation}\label{INEQ}
R_w(\mathbf{P}_1,\mathbf{P}_2\ldots, \mathbf{P}_k)\le R(\mathbf{P}_1,\mathbf{P}_2\ldots, \mathbf{P}_k)    
\end{equation} 
from the definitions. 
If all the $\mathbf{P}_i$'s are chains, then equality in (\ref{INEQ}) holds.
Observe that the colorings we used to derive the lower bounds for $R_k(\mathbf{V}_{1,1})$ and $R(\mathbf{V}_{m,m},\mathbf{V}_{n,n})$ also avoid the weak $V$-shaped posets. So $R_k(\mathbf{V}_{1,1})$ and $R(\mathbf{V}_{m,m},\mathbf{V}_{n,n})$ are the examples for which the equality in (\ref{INEQ}) holds. 
On the other hand, we have $R_k(\mathbf{V}_{m,n})=nk+1$ for $m<n$ which implies $R_2(\mathbf{V}_{1,2})=5$. 
However, for the weak version we have 
$R_w(\mathbf{V}_{1,2},\mathbf{V}_{1,2})=4$.
This is because for every two coloring of $\mathbf{B}_4$, the nonempty subsets of $[4]$ that have the same color as $\varnothing$ must be a chain or an antichain. Thus, we can find a copy of monochromatic $\mathbf{B}_2$ in the subposet of $\mathbf{B}_4$ formed by removing all elements of the same color as $\varnothing$. Note that $\mathbf{B}_2$ contains $\mathbf{V}_{1,2}$ as a weak subposet. So $R_w(\mathbf{V}_{1,2},\mathbf{V}_{1,2})\le 4$.
Meanwhile, it is easy to color $\mathbf{B}_3$ by two colors so that no monochromatic weak $\mathbf{V}_{1,2}$ exists. In general, given $\mathbf{P}_i=\mathbf{V}_{m_i,n_i}$ with $1\le m_i< n_i$ for $1\le i\le k$, one can show 
$
R_w(\mathbf{P}_1,\mathbf{P}_2,\ldots, \mathbf{P}_k)
<R(\mathbf{P}_1,\mathbf{P}_2,\ldots, \mathbf{P}_k)$. We believe that Theorem~\ref{MT} can be generalized to more posets and the following is true.

\begin{conjecture} For $\mathbf{P}_{i}=\mathbf{V}_{m_i,m_i}$, 
\[
R_w(\mathbf{P}_1,\mathbf{P}_2,\ldots, \mathbf{P}_k)
=R(\mathbf{P}_1,\mathbf{P}_2,\ldots, \mathbf{P}_k)
=m_1+m_2+\cdots+m_k+1.\]
\end{conjecture}

\end{document}